\documentclass[reqno,12pt]{amsart}

\usepackage{amsmath}
\usepackage{amscd}
\usepackage{amssymb}
\usepackage{enumerate}
\usepackage{amsfonts}
\usepackage{graphicx}
\usepackage[all]{xy}
\usepackage{mathrsfs}
\usepackage[hypertex]{hyperref}

\usepackage{bbm}

\newcommand{\e}{\mathbbm{1}}

\newcommand{\limto}{{\displaystyle\lim_{\longrightarrow}}}
\newcommand{\rightlim}{\mathop{\limto}}

\newcommand{\leftlim}{\mathop{\displaystyle\lim_{\longleftarrow}}}
\newcommand{\limfromn}{\leftlim\limits_{\raise3pt\hbox{$n$}}}
\newcommand{\limton}{\rightlim\limits_{\raise3pt\hbox{$n$}}}

\newcommand{\rightlimit}[1]{\mathop{\lim\limits_{\longrightarrow}}\limits%
                   _{\raise3pt\hbox{$\scriptstyle #1$}}}
\newcommand{\leftlimit}[1]{\mathop{\lim\limits_{\longleftarrow}}\limits%
                   _{\raise3pt\hbox{$\scriptstyle #1$}}}

\numberwithin{equation}{section}

\newcommand{\rar}[1]{\stackrel{#1}{\longrightarrow}}
\newcommand{\lar}[1]{\stackrel{#1}{\longleftarrow}}
\newcommand{\xrar}[1]{\xrightarrow{#1}}

\newcommand{\iso}{\buildrel{\sim}\over{\longrightarrow}}

\newcommand{\into}{\hookrightarrow}

\newcommand{\al}{\alpha}
\newcommand{\be}{\beta}
\newcommand{\ga}{\gamma}
\newcommand{\Ga}{\Gamma}
\newcommand{\de}{\delta}

\newcommand{\eps}{\epsilon}

\newcommand{\te}{\theta}

\newcommand{\vp}{\varphi}
\newcommand{\vt}{\vartheta}

\newcommand{\bA}{{\mathbb A}}

\newcommand{\bD}{{\mathbb D}}

\newcommand{\bQ}{{\mathbb Q}}
\newcommand{\bR}{{\mathbb R}}

\newcommand{\bZ}{{\mathbb Z}}

\newcommand{\cA}{{\mathcal A}}

\newcommand{\cC}{{\mathcal C}}
\newcommand{\cD}{{\mathcal D}}
\newcommand{\cE}{{\mathcal E}}

\newcommand{\cM}{{\mathcal M}}
\newcommand{\cN}{{\mathcal N}}

\newcommand{\cS}{{\mathcal S}}

\newcommand{\cY}{{\mathcal Y}}

\newcommand{\sD}{{\mathscr D}}

\newcommand{\sL}{{\mathscr L}}

\newcommand{\fA}{{\mathfrak A}}

\newcommand{\fM}{{\mathfrak M}}

\newcommand{\Isom}{\operatorname{Isom}}

\newcommand{\End}{\operatorname{End}}
\newcommand{\Hom}{\operatorname{Hom}}

\newcommand{\Aut}{\operatorname{Aut}}
\newcommand{\Spec}{\operatorname{Spec}}

\newcommand{\id}{\operatorname{id}}
\newcommand{\Id}{\operatorname{Id}}

\newcommand{\Ad}{\operatorname{Ad}}

\newcommand{\HOM}{\underline{\operatorname{Hom}}}

\newcommand{\tens}{\otimes}

\newcommand{\st}{\,\big\vert\,}

\newcommand{\sbr}{\smallbreak}
\newcommand{\mbr}{\medbreak}
\newcommand{\bbr}{\bigbreak}

\newtheorem{thm}{Theorem}[section]
\newtheorem{cor}[thm]{Corollary}
\newtheorem{lem}[thm]{Lemma}

\newtheorem{prop}[thm]{Proposition}

\theoremstyle{remark}
\newtheorem{rem}[thm]{Remark}
\newtheorem{rems}[thm]{Remarks}
\newtheorem{example}[thm]{Example}

\newtheorem{defin}[thm]{Definition}

\newcommand{\ql}{\overline{\bQ}_\ell}

\newcommand{\sets}{\cS{}ets}
\newcommand{\GV}{Gro\-then\-dieck-Ver\-dier\ }

\title[A duality formalism in the spirit of Grothendieck and Verdier]{A duality formalism in the spirit of \\ Grothendieck and Verdier}

\author{Mitya Boyarchenko \and Vladimir Drinfeld}

\thanks{M.B. (corresponding author) was supported by the NSF grant DMS-1001769. \\ V.D. was supported by the NSF grant DMS-1001660. \\
\textit{Addresses}: M.B.: Department of Mathematics, University
of Michigan, Ann Arbor, MI 48109. \\ V.D.: Department of
Mathematics, University of Chicago, Chicago, IL 60637. \\
\textit{E-mail}: {\tt mitya@math.uchicago.edu} (M.B.), \ {\tt
drinfeld@math.uchicago.edu} (V.D.)}

\begin{document}

\begin{abstract}
We study monoidal categories that enjoy a certain weakening of the rigidity property, namely,
the existence of a dualizing object in the sense of Grothendieck and Verdier. We call them
Grothendieck-Verdier categories. (They have also been studied in the literature under the name ``$*$-autonomous categories.'')
Notable examples include the derived category of constructible sheaves 
on a scheme
(with respect to tensor product) as well as the derived and equivariant derived categories of
constructible sheaves on an algebraic group (with respect to convolution).

We show that the notions of pivotal category and ribbon category, which are well known in the setting of rigid monoidal categories, as well as certain standard results associated with these notions, have natural analogues in the world of \GV categories.
\end{abstract}



\maketitle

\setcounter{tocdepth}{1}

\tableofcontents


\section*{Introduction}

\subsection{Main definitions}
\begin{defin}\label{def:dualizing}
An object $K$ in a  monoidal category $\cM$ is said to be \emph{dualizing}  if
for every $Y\in\cM$ the functor $X\mapsto\Hom (X\otimes Y,K )$ is
representable by some object $DY\in\cM$ and the contravariant
functor $D:\cM\rar{} \cM$ is an antiequivalence. $D$ is called the
\emph{duality functor} with respect to $K$.
\end{defin}

\begin{rem}
By Proposition \ref{p:dualizings} below,
if a dualizing object exists then it is unique up to tensoring by an
invertible object.
\end{rem}

\begin{defin}\label{def:GV}
A \emph{Grothendieck-Verdier category} is a pair $(\cM,K)$, where $\cM$ is a monoidal category
and $K\in\cM$ is a dualizing object.
\end{defin}

Some examples of \GV categories are given in \S\ref{ss:first_examples} below.

\begin{rems}    \label{r:abuses}
\begin{enumerate}[(1)]
\item If $(\cM,K)$ is a \GV category then $D:\cM\rar{}\cM$ will always denote
the corresponding duality functor.
 \sbr
\item By an abuse of language we will sometimes say ``Grothendieck-Verdier category $\cM$'' instead of
``Grothendieck-Verdier category $(\cM,K)$''.
\end{enumerate}
\end{rems}

\begin{defin}\label{def:r-category}
A monoidal category $\cM$ is said to be an {\em r-category\,} if
the unit object $\e\in\cM$ is dualizing.
\end{defin}

So any r-category can be considered as a Grothendieck-Verdier category with $K=\e$.
The letter `r' in the name ``r-category" is related to the words ``rigid'' and ``regular'', see
Examples~\ref{example:dual-gen1}--\ref{example:dual-gen2} below.

\subsection{Main examples}   \label{ss:first_examples}
Below we give some examples of Grothendieck-Verdier categories. More examples of
such categories can be found in \S\ref{s:examples} and in the works by M.~Barr, who studied them under the name of {\em $\ast$-autonomous categories\,}  (e.g., see  \cite{Barr79,Barr95,Barr96,Barr99}).

\begin{example}\label{example:GVoriginal}
Let $\cM=(\sD (X),\otimes)$, where $X$ is a scheme of finite type over a field $k$
and $\sD(X)=D^b_c(X,\ql)$ is the bounded derived category of constructible $\ell$-adic
sheaves on $X$, defined as in \cite{Jannsen,Ekedahl}.
Let $K_X\in\sD (X)$ be the dualizing complex. Then $(\cM ,K_X)$ is a
Grothendieck-Verdier category. In this case $D$ is the usual Verdier duality
functor $ \bD_X$.
\end{example}

\begin{example}\label{example:Kashiwara-Schapira}
In \cite{KS} M.~Kashiwara and P.~Schapira introduce three variants of the
category $\cM$ from Example~\ref{example:GVoriginal}
in which sheaves are considered with respect to a usual topology
(rather than a Grothendieck topology). More precisely, let $X$ be either a
locally finite simplicial complex or a real-analytic manifold, or a complex-analytic one.
In each of these situations they introduce in  \cite[ch.~8]{KS} a notion of constructibility for sheaves of
abelian groups on $X$ so that the bounded derived category of constructible sheaves becomes
a Grothendieck-Verdier category.
\end{example}

\begin{example} \label{example:dual-gen1}
Any rigid monoidal category\footnote{The definition of rigidity is recalled in \S\ref{ss:rigidity}, see
Definition~\ref{d:left-dual}.} is an r-category. The next example (or the elementary
Example~\ref{example:earlyGrothendieck}) shows that the converse is false.
\end{example}

\begin{example}   \label{example:dual-gen2}
Let $X$ be a smooth scheme of pure dimension $d$ over a field $k$. Then the monoidal category
$(\sD (X),\otimes )$ is an r-category and $D: \sD (X)\to\sD (X)$ is the functor
$N\longmapsto( \bD_X  N)[-2d](-d)$. If $d>0$, then $(\sD (X),\otimes )$ is \emph{not rigid} because
$D(M_1\tens M_2)\not\cong D(M_2)\tens D(M_1)$ for some $M_1,M_2\in\sD (X)$.
For instance, let $M_1=M_2=i_*\ql$, where $i:\Spec k\into X$ is a point. Then
$D(M_1\tens M_2)= D(i_*\ql)=i_*\ql [-2d](-d)$, while $D(M_2)\tens D(M_1)=i_*\ql [-4d](-2d)$.
\end{example}

\begin{example}\label{example:dual-gen4}
Let $G$ be a group scheme of finite type over a field $k$. We define the equivariant derived category of $G$ as $\sD_G(G):=D^b_c\bigl((\Ad G)\backslash G,\ql\bigr)$ (i.e., $\sD_G(G)$ is the bounded derived category of the quotient stack for the conjugation action of $G$ on itself \cite{Las-Ols06}).
The monoidal categories $\sD (G)$ and $\sD_G(G)$ equipped with the functor of convolution with compact
support are r-categories with $D$ being the functor $\bD_G^-=\bD_G\circ\iota^*=\iota^*\circ\bD_G$, where
$\bD_G$ is the Verdier duality functor on $G$ and $\iota:G\rar{}G$ is given by $g\mapsto g^{-1}$.
The proof is straightforward and easy, see \cite[Lemma A.10]{foundations}.
The monoidal category $\sD_G(G)$ has a canonical braided structure, see \cite[Definition A.43]{foundations}.
\end{example}

\subsection{Subject of this work}
Our goal is to establish some general facts about \GV categories, which are well known in the
case of symmetric \GV categories or in the case of arbitrary rigid monoidal categories.
The proofs are not always straightforward generalizations of existing ones.

\mbr

For instance, if $\cM$ is a rigid monoidal category then the well known monoidal structure on the functor
$D^2:\cM\rar{}\cM$ is usually defined via the canonical isomorphism
$D(X\tens Y)\rar{\simeq} DY\tens DX$. In an arbitrary \GV category (or even an arbitrary r-category)
$D(X\tens Y)$ is, in general, not isomorphic to $DY\tens DX$ (see Example~\ref{example:dual-gen2}).
Nevertheless, the functor $D^2$ has a canonical monoidal structure, see~\S\ref{s:D^2monoidal}.

\mbr

Here is another example.
It is well known that the set of twists\footnote{The notion of twist is recalled in \S\ref{ss:JS}, see
Definition~\ref{d:JS} and Remark~\ref{r:twists}(i).} on a rigid braided category $\cM$
is equipped with a canonical involution: namely, if $\theta\in\Aut\Id_{\cM}$ is a twist then the automorphism $\theta'\in\Aut\Id_{\cM}$ defined by $\te'_X=D^{-1}(\te_{DX})$ is also a twist.
(The fixed points of this involution are called \emph{ribbon structures}.) For arbitrary r-categories\footnote{For \GV categories the statement has to be slightly modified, see
Proposition~\ref{p:involution}.} it is still true that $\theta'$ is a twist (see Proposition~\ref{p:involution}
and Remark~\ref{r:involution}), but the proof has to be modified.

\subsection{An $\infty$-categorical perspective (after J.~Lurie)}  \label{ss:Lurie}
This subsection is informal. We hope that somebody will develop these ideas rigorously and systematically.

\subsubsection{}
There is a general notion of $\cE_n$-category, i.e., an $(\infty ,1)$-category\footnote{A
$(\infty ,1)$-category is an $\infty$-category in which all $m$-morphisms are invertible for $m>1$.}
with an action of the little $n$-disk operad $\cE_n$. If $n=1$ and $n=2$ one gets, respectively, the notions of monoidal and braided $(\infty ,1)$-category.

An object of a monoidal $(\infty ,1)$-category is said to be \emph{dualizing} if it is dualizing in its
homotopy category (which is a usual monoidal category). Thus one has a notion of \GV
$(\infty ,1)$-category. Since $\cE_1\subset\cE_n$ one has a notion of \GV
$\cE_n$-category for each $n\ge 1$.

\begin{example}
The \GV categories from
Example~\ref{example:GVoriginal} and Examples \ref{example:dual-gen2}-\ref{example:dual-gen4} have natural $(\infty ,1)$-categorical ``refinements". In particular, the ``refinement" of
the category $\sD (G)$ from Example~\ref{example:dual-gen4} is an $\cE_1$-category and
the ``refinement" of $\sD_G(G)$ is an $\cE_2$-category.
\end{example}

\subsubsection{}
As explained to us by J.~Lurie, he expects (or at least, he does not exclude) that the results of
\S\ref{s:D^2monoidal}-\ref{s:ribbon}
can be generalized to this setting and interpreted in terms of a certain canonical action of the topological group\footnote{An $\infty$-groupoid is essentially the same as a topological space, so it can be acted upon by a topological group.} $O(n+1)$ on the $\infty$-groupoid of \GV $\cE_n$-categories, whose restriction to
$O(n)\subset O(n+1)$ comes from the obvious action of $O(n)$ on the operad $\cE_n$.
This would be very interesting.
In Example 4.4.14 of \cite{Lurie} Lurie sketches a construction of the $O(n+1)$-action on the space of
\emph{rigid}  $\cE_n$-categories.

\mbr

Most of the results of our \S\S\ref{s:D^2monoidal}-\ref{s:ribbon}
(and their well known prototypes in the rigid case)
can be interpreted from this perspective. For instance, the fact that any \GV category $\cM$ has a canonical
auto-equivalence (namely, $D^2:\cM\iso\cM$) is related to the canonical generator of
$\pi_1 (O(2))$, and the fact that for any braided \GV category $\cM$ one has a canonical monoidal isomorphism $D^4\rar{\simeq}\Id_{\cM}$ (see \S\ref{sss:D^4})
is related to the equality $\pi_1 (O(3))=\bZ/2\bZ$.

\subsubsection{}
A related idea is to regard an $\cE_n$-category $\cM$ as a fiber of a certain \emph{local system}
of $(\infty ,1)$-categories, $\fM$, over the sphere $S^n$. To define $\fM$, note that an
$\cE_n$-category $\cM$ has not a single tensor product but rather a family of tensor products\footnote{In the familiar case $n=1$ a monoidal category has two tensor products: the ``original" one and the opposite one. In the \GV case the corresponding duality functors are $D$ and
$D^{-1}$.} $\tens_{\omega}$ parameterized by $\omega\in S^{n-1}$. Accordingly, in a \GV
$\cE_n$-category  $\cM$ one has not a single duality functor $D$ but rather
a family of equivalences $D_{\omega}:\cM^{\circ}\iso\cM$,
$\omega\in S^{n-1}$ (here $\cM^{\circ}$ is the dual $(\infty ,1)$-category). To construct $\fM$, represent $S^n$ as the union of hemispheres $S^n_{\pm}$, consider the constant sheaf on $S^n_+$ (resp. $S^n_-$) with fiber $\cM$ (resp. $\cM^{\circ}$) and glue them together using
$D_{\omega}$,  $\omega\in S^{n-1}=S^n_+\cap S^n_-\,$. Note that in general, $\fM$ is \emph{not} a local system of $\cE_n$-categories; in other words, the action of the loop space $\Omega S^n$ on $\cM$ defined by $\fM$ does \emph{not} preserve the $\cE_n$-structure on $\cM$.
E.g., if  $\cM$ is a braided \GV category then the image in $\Aut (\Id_{\cM})$ of the generator of
$\pi_1(\Omega S^2)=\pi_2(S^2)$ equals the automorphism $C_{\cM}\in\Aut (\Id_{\cM})$ from \S\ref{ss:JS}, which is a \emph{double-twist} in the sense of Definition~\ref{d:twists} and Remark~\ref{r:twists}(i) rather than a monoidal automorphism.

\subsection{Structure of the article}
We already defined the main objects of our study, \GV categories and r-categories, and remarked that every rigid monoidal category is an r-category. We begin the article by giving some basic properties of \GV categories in \S\ref{s:first-properties} and some further examples of (non-rigid) \GV categories and r-categories in \S\ref{s:examples}. In \S\ref{s:rigidity} we characterize rigid monoidal categories as r-categories satisfying a certain additional property.

\mbr

We devote \S\S\ref{s:D^2monoidal}--\ref{s:ribbon} to generalizations of certain well-known results and constructions involving rigid monoidal categories to the setting of \GV categories. In particular,
in \S\ref{s:D^2monoidal} we define a canonical monoidal structure on the square of the duality functor for an arbitrary \GV category. In \S\ref{s:pivotal} we define and study pivotal structures on \GV categories. In \S\ref{s:braidedGV} we study braided \GV categories. In particular, we prove that for any such category the square of the duality functor is 
braided and its fourth power is canonically isomorphic to the identity functor.
In \S\ref{s:pivotal-braided} we analyze the relation between pivotal structures and twists on a braided \GV category. This leads us to introducing in \S\ref{s:ribbon} the notion of a ribbon \GV category
(which specializes to the usual notion in the rigid case).

\mbr

We end the first part of the article by answering in \S\ref{s:relation} the question of which \GV categories can be realized as Hecke subcategories of r-categories.

\mbr

The second part of the article (\S\S\ref{s:proof-rigidity}--\ref{s:GV_to_r}) is devoted to the proofs that are too long and/or too technical to be included 
into the first part (namely, the proofs of Propositions \ref{p:Denis}, \ref{p:rigidity}, \ref{p:DD'}, \ref{p:2pivotal}, \ref{p:monoidal-iso-id-DD} and \ref{p:inverse-construction}, as well as Lemma \ref{l:braided-GV-category-phi-pm}).

\subsection{Acknowledgments}
We thank J.~Ayoub, A.~Beilinson, P.~Etingof, D.~Gaitsgory, E.~Jenkins, and especially J.~Lurie for useful discussions and advice. We also thank the referee for helpful suggestions and for informing us about the articles \cite{day-street} and \cite{egger-mccurdy}.

\vfill\newpage

\part{Formulations and easy proofs}

\section{First properties of Grothendieck-Verdier categories}\label{s:first-properties}

\subsection{Some canonical isomorphisms}

\begin{rems}\label{r:r}
\begin{enumerate}[(i)]
\item By definition, in any  Grothendieck-Verdier category $\cM$ one has an isomorphism
\begin{equation}  \label{e:duality1}
\Hom (X\otimes Y, K )\rar{\simeq}\Hom (X,DY)
\end{equation}
functorial in $X,Y\in\cM$. Since $D$ is an antiequivalence the right-hand side
of \eqref{e:duality1} identifies with $\Hom (Y,D^{-1}X)$. So one
also has an isomorphism
\begin{equation}  \label{e:duality2}
\Hom (X\otimes Y, K )\rar{\simeq}\Hom (Y,D^{-1}X)
\end{equation}
functorial in $X,Y\in\cM$. Thus a Grothendieck-Verdier category equipped with the opposite
tensor product is still a Grothendieck-Verdier category, {\em but $D$ gets replaced by $D^{-1}$.}
 \mbr
\item
By \eqref{e:duality2}, in any Grothendieck-Verdier category $\cM$ one has a functorial isomorphism
$\Hom (D^2Y\otimes X, K )\rar{\simeq}\Hom (X,DY)$.
Combining it with \eqref{e:duality1} one gets a functorial isomorphism
\begin{equation}  \label{e:duality3}
g:\Hom (X\otimes Y, K )\rar{\simeq}\Hom (D^2Y\otimes X, K ), \quad X,Y\in\cM .
\end{equation}
Equivalently, $g$ is characterized by the commutativity of the diagram
\begin{equation}   \label{e:duality4}
\xymatrix{
\Hom (X\otimes Y, K )  \ar[d] \ar[rr]^{g} & &
\Hom (D^2Y\otimes X, K )\ar[d]\\
 \Hom (X,DY) \ar[rr]^{D} & &
 \Hom (D^2Y,DX)}
\end{equation}
whose vertical arrows come from \eqref{e:duality1}.
 \mbr
\item
In any Grothendieck-Verdier category there exist right and left\footnote{In this article we do
not have to decide which of the two internal $\Hom$'s
defined by \eqref{e:internal_Hom2b} and \eqref{e:internal_Hom1b} should be called ``left.''
We prefer the convention that $\HOM$ is the right internal $\Hom$ and $\HOM '$ is the left one.
Reason: by Proposition~\ref{p:criterion-for-duality}, the right rigid dual of $X$ (if it exists) equals
$\HOM (X,\e )$. Note that the functor of \emph{left} multiplication by $X$ is adjoint to the functor
$\HOM (X,?)$, which we would like to call the  \emph{right} internal $\Hom$.
We think this is acceptable.} internal $\Hom$'s. More precisely, if one sets
\begin{equation}   \label{e:internal_Hom2a}
\HOM(X,Z)=D^{-1}(DZ\tens X)
\end{equation}
and
\begin{equation}  \label{e:internal_Hom1a}
\HOM'(Y,Z)=D(Y\tens D^{-1}Z),
\end{equation}
then \eqref{e:duality1} and \eqref{e:duality2} yield
functorial isomorphisms
\begin{equation}   \label{e:internal_Hom2b}
\Hom (X\otimes Y,Z)\rar{\simeq}\Hom (Y, \HOM(X,Z))
\end{equation}
and
\begin{equation}  \label{e:internal_Hom1b}
\Hom (X\otimes Y,Z)\rar{\simeq}\Hom (X, \HOM'(Y,Z)).
\end{equation}
 \mbr
\item From \eqref{e:duality1} and  \eqref{e:duality2} one gets canonical isomorphisms
\begin{equation}   \label{e:D1}
D\e\rar{\simeq} K, \quad\quad D^{-1}\e\rar{\simeq} K.
\end{equation}
and therefore canonical isomorphisms
\begin{equation}  \label{e:D^2(1)}
\e\rar{\simeq} D^2\e ,
\end{equation}
\begin{equation}  \label{e:D^2K}
K\rar{\simeq} D^2K,
\end{equation}
where \eqref{e:D^2K} is the composition $K\rar{\simeq} D\e\rar{\simeq} D^2D^{-1}\e\rar{\simeq} D^2K$.
 \mbr
\item The inverse of \eqref{e:D^2K} equals the image of
$\id_K\in\Hom(\e\tens K,K)$ under the isomorphism $\Hom(\e\tens K,K)\rar{\simeq}\Hom(D^2K\tens\e,K)$
coming from \eqref{e:duality3}. It is easy to check this using diagram \eqref{e:duality4} for $X=\e$ and $Y=K$.
\end{enumerate}
\end{rems}

\subsection{Uniqueness of dualizing objects} Let us recall the following

\begin{defin}
If $\cM$ is a monoidal category, an object $X\in\cM$ is said to be \emph{invertible} if there exists an object $Y\in\cM$ such that $X\tens Y\cong\e\cong Y\tens X$.
\end{defin}

\begin{prop}    \label{p:dualizings}
Let $(\cM ,K)$ be a Grothendieck-Verdier category.
\begin{enumerate}[$($i$)$]
\item The functor $\sL\mapsto D\sL=K\tens\sL^{-1}$ is an antiequivalence between the full subcategory
of invertible objects $\sL\in\cM$ and the full subcategory of dualizing objects.
 \sbr
\item The same is true for the functor $\sL\mapsto D^{-1}\sL=\sL^{-1}\tens K$.
 \sbr
\item If $\sL\in\cM$ is invertible then so is $D^2\sL$ and one has a canonical isomorphism
 $K\tens\sL^{-1}\rar{\simeq} (D^2\sL)^{-1}\tens K$.
\end{enumerate}
\end{prop}

\begin{proof}
By \eqref{e:internal_Hom1a} and \eqref{e:internal_Hom1b},
an object $Z\in\cM$ is dualizing if and only if the functor
$Y\mapsto Y\tens D^{-1}Z$ is an equivalence. This means that
$Z=D\sL$, where $\sL$ is invertible. In this case $D\sL=K\tens\sL^{-1}$ by \eqref{e:duality1}.
We have proved (i). To prove (ii), use \eqref{e:internal_Hom2a}, \eqref{e:internal_Hom2b},
and \eqref{e:duality2} instead of
\eqref{e:internal_Hom1a}, \eqref{e:internal_Hom1b}, and \eqref{e:duality1}.

\mbr

By (ii), $K\tens\sL^{-1}=D\sL$ can also be written as $\widetilde{\sL}^{-1}\tens K=D^{-1}\widetilde{\sL}$
for some invertible $\widetilde{\sL}\in\cM$. Since $D^{-1}\widetilde{\sL}=D\sL$ we have
$\widetilde{\sL}=D^2\sL$.
\end{proof}

\begin{rem}     \label{r:commutationrule}
Proposition~\ref{p:DD'} below yields a canonical isomorphism
\begin{equation}  \label{e:D^2inverse}
(D^2\sL)^{-1}\rar{\simeq} D^2(\sL ^{-1}).
\end{equation}
\end{rem}

\subsection{Invertibility and rigidity of $K$}    \label{sss:rigidity_of_K}
We learned the following statement from Dennis Gaitsgory.
He also explained to us how it can be applied to studying the derived categories
of $\cD$-modules on certain algebraic stacks.

\begin{prop}   \label{p:Denis}
A dualizing object of a monoidal category is invertible if and only if it is rigid in the sense of Definition~\ref{d:left-dual}.
\end{prop}

See \S\ref{sss:inv&rig} for the proof.

\section{More examples of Grothendieck-Verdier categories}
\label{s:examples}
We already gave some examples in \S\ref{ss:first_examples}. More examples are below.

\begin{example}\label{example:dual-gen3}
As far as we understand, O.~Gabber recently proved that $(\sD (X),\otimes )$ is an
r-category for any excellent regular scheme $X$ over $\bZ [\ell^{-1}]$ (not necessarily of finite type).
\end{example}

\begin{example}     \label{example:dual-gen5}
Here is a generalization of Example~\ref{example:dual-gen4}. Suppose we have a groupoid in the category of schemes of finite type over a field $k$. Let $\Ga$ denote its ``scheme of morphisms,'' and let $X$ denote its ``scheme of objects,'' so one has the source and target maps $s,t:\Ga\rar{}X$,
the unit $1:X\rar{}\Ga$, the inversion $\iota :\Ga\rar{\simeq}\Ga$ and the product $\mu :\Ga\times_X\Ga\rar{}\Ga$,
where
\[
\Ga\times_X\Ga:=\{ (\ga_1,\ga_2)\in\Ga\times\Ga \st s(\ga_1)=t(\ga_2)\}.
\]
For $M_1,M_2\in\sD(\Ga )$ set $M_1*M_2:=\mu_!\bigl(p_1^* M_1\tens p_2^* M_2\bigr)$,
where $p_1,p_2:\Ga\times_X\Ga\rar{}\Ga$ are the projections. Then $\sD (\Ga )$ becomes
a monoidal category with unit object $\e =1_!\ql\,$. Define $K\in\sD (\Gamma )$ by $K:=1_*K_X$,
where $K_X\in\sD (X)$ is the dualizing complex. Then $(\sD (\Gamma ),K)$ is a
Grothendieck-Verdier category with duality functor
$\bD_{\Ga}^-:=\bD_{\Ga}\circ\iota^*=\iota^*\circ\bD_{\Ga}$. Moreover, if an algebraic group $H$
acts on $(\Ga ,X,s,t,\mu )$ then $(\sD_H(\Gamma ),K)$ is a
Grothendieck-Verdier category, where $\sD_H(\Gamma):=D^b_c(H\backslash\Ga,\ql)$ is the bounded derived category \cite{Las-Ols06} of the quotient stack $H\backslash\Ga$. (The proof of these assertions is very similar to the proof of
\cite[Lemma A.10]{foundations},
so we omit it.) If $X$ is smooth and the embedding
$1:X\rar{}\Ga$ is closed then $1_*K_X=1_!K_X$ is an invertible object of $\sD (\Gamma )$, so
$\sD (\Gamma )$ is an $r$-category. Note that if $\Ga =X\times X$ then $1:X\rar{}\Ga$ is the
diagonal embedding, so the above closedness condition means that $X$ is separated.
\end{example}

The following elementary example of a non-rigid r-category is closely related to the works of
Grothendieck in functional analysis. We learned this example from \cite{Barr79}.

\begin{example}   \label{example:earlyGrothendieck}
Let $\cM$ be the category of finite-dimensional normed vector spaces over $\bR$ with
morphisms being linear operators of norm $\le 1$.
For $V,W\in\cM$ define $V\tens W$ to be the tensor product of vector spaces $V$ and $W$
equipped with the maximal norm such that $||v\tens w||\le ||v||\tens ||w||$
for all $v\in V$, $w\in W$. The symmetric monoidal category $\cM$ is an r-category with
$D$ being the usual dual of a normed vector space. But $\cM$ is not rigid. In fact, an object $V\in\cM$
is rigid if and only if $\dim V\le 1$ (to prove the ``only if\," statement, note that if $V$ is rigid then the composition of the coevaluation map $\e\rar{} V\tens V^*$ and the evaluation map $V\tens V^*\rar{}\e$ has norm $\le 1$, where $V^*$ denotes the dual of $V$). By definition, $DV\tens W$ identifies with
 the space $\Hom (V,W)$ equipped with the {\em nuclear norm.} On the other hand,
one easily shows that $D(V\tens DW)$ is the space $\Hom (V,W)$ equipped with the
{\em operator norm.}
\end{example}

One can obtain more examples of \GV categories
using Lemma \ref{l:GV-Hecke-subcategory} below. To formulate it, we need the following
definition from \cite[\S2]{foundations}.

\begin{defin}   \label{d:closed_idemp}
A morphism $\pi :\e\rar{}e$ in
$\cM$ is said to be an \emph{idempotent arrow} if both
morphisms $\pi\tens\id_e:\e\tens e\rar{}e\tens e$ and $\id_e\tens\pi:\e\tens e\rar{}e\tens e$
are isomorphisms. An object $e$ of a monoidal category $\cM$ is said to
be a \emph{closed idempotent}\footnote{In the situation of Definition~\ref{d:closed_idemp} one
has $e\tens e\simeq e$, so the name ``idempotent" is justified. The adjective ``closed" is due to the fact that closed idempotents in the monoidal category $\cM=\sD (X)$ from
Example~\ref{example:GVoriginal}  bijectively correspond to closed subsets  $Y\subset X$.
Namely, such $Y$ defines a closed idempotent $e=(\ql )_Y\in\sD (X)$, and the corresponding
monoidal category \eqref{e:eMe} identifies with ${\mathscr D} (Y)$.} if there exists an idempotent arrow $\e\rar{}e$.
\end{defin}

\begin{rem}
In the situation of Example~\ref{example:dual-gen4} with $G$ \emph{unipotent} the categories
$\sD (G)$ and $\sD_G(G)$ have many closed idempotents, see \cite[\S1]{foundations}
(especially \S 1.11 and Theorems 1.41(a), 1.49(c) from \cite{foundations}).
\end{rem}

If $e\in\cM$ is a closed idempotent, we set
\begin{equation}   \label{e:eMe}
e\cM e: = \bigl\{ X\in\cM \st X\cong e\tens Y\tens e \text{ for some } Y\in\cM \bigr\}.
\end{equation}
The tensor product of objects of $e\cM e$ clearly belongs to $e\cM e$.
Equipped with this tensor product, $e\cM e$ is a monoidal category with unit object $e$,
see \cite[Lem\-ma~2.18]{foundations}. More precisely, an idempotent arrow
$\pi :\e\rar{}e$ defines a structure of unit object on $e$, see Lemma~\ref{l:idempotents-auxiliary}(b)
below.

\begin{defin}   \label{d:Hecke}
We call $e\cM e\subset\cM$ the \emph{Hecke subcategory} of $\cM$ defined
by $e$.
\end{defin}

\begin{lem}\label{l:GV-Hecke-subcategory}
Let $(\cM,K)$ be a Grothendieck-Verdier category, and
let $e\in\cM$ be a closed idempotent such that $D^2e\cong e$. Then
$De$ is a dualizing object of the monoidal category $e\cM e$, so
$(e\cM e, De)$ is a Grothendieck-Verdier category. In fact, $D (e\cM e)=e\cM e$, and
the duality functor for $(e\cM e, De)$ is isomorphic to the restriction of $D$ to $e\cM e$.
\end{lem}

See \cite[Lemma A.50]{foundations} for a more precise version of Lemma~\ref{l:GV-Hecke-subcategory} and a proof.

\mbr

In \S\ref{s:relation} below we answer the following question:
which \GV categories can be realized as Hecke subcategories of r-categories? For example, their class includes all additive \GV categories (apply Proposition~\ref{p:inverse-construction} to $f=0$).

\section{Rigidity in r-categories}  \label{s:rigidity}

\subsection{The second tensor product in an r-category} \label{sss:odot}
Let $\cM$ be an r-category. Define a new monoidal
structure\footnote{One could also consider the monoidal structure given
by $(X,Y)\longmapsto D(D^{-1}Y\otimes D^{-1}X)$, but Proposition~\ref{p:DD'} below
allows us to identify it with $X\odot Y$.} $\odot:\cM\times\cM\to\cM$
by
\begin{equation}  \label{e:otherproduct}
X\odot Y:=D^{-1}(DY\otimes DX).
\end{equation}
Let us define a morphism
\begin{equation}  \label{e:forgetsupp}
X\otimes Y\to X\odot Y,  \quad X,Y\in\cM .
\end{equation}
To this end, note that by \eqref{e:duality1} we have canonical
morphisms $DX\otimes X\to\e$ and $DY\otimes Y\to\e$. So we get a
morphism $DY\tens DX\tens X\tens Y\to\e$, and by \eqref{e:duality2}
this is the same as a morphism
$X\tens Y\to D^{-1}(DY\otimes DX)=X\odot Y$.
Clearly \eqref{e:forgetsupp} is functorial in $X,Y$.

\begin{lem}
The morphism \eqref{e:forgetsupp} is compatible with the associativity constraints for
$\otimes$ and $\odot$.
\end{lem}

\begin{proof}

We have to show that the morphisms
$$f:X_1\otimes X_2\otimes X_3=(X_1\otimes X_2)\otimes X_3\to (X_1\odot X_2)\odot X_3=
X_1\odot X_2\odot X_3$$ 
and
$$g:X_1\otimes X_2\otimes X_3=X_1\otimes (X_2\otimes X_3)\to X_1\odot (X_2\odot X_3)=X_1\odot X_2\odot X_3$$ 
coming from \eqref{e:forgetsupp} are equal. By \eqref{e:duality1}, we have canonical
morphisms $DX_i\otimes X_i\to\e$ and therefore a morphism $h:DX_3\tens DX_2\tens DX_1\tens
X_1\tens  X_2\tens X_3\to\e$. By \eqref{e:duality2}, this is the same as a morphism
$h':X_1\otimes X_2\otimes X_3\to X_1\odot X_2\odot X_3$. Both $f$ and $g$ equal $h'$.
\end{proof}

\begin{rem}
Since $\cM$ is an r-category, $\e$ is a unit object for both $\tens$ and $\odot$. It is not hard to check that the morphism \eqref{e:forgetsupp} is compatible with the unit constraint for $\tens$ and $\odot$.
\end{rem}

\begin{example}   \label{example:forgetsupp}
In the situation of Example~\ref{example:dual-gen4} the monoidal functor \eqref{e:otherproduct}
is the convolution without compact support and \eqref {e:forgetsupp} is the usual morphism.
\end{example}

\subsection{Rigidity in $r$-categories}
Let us discuss the relation between the functor $D:\cM\rar{}\cM$ and the notion of rigid duality from
Definition~\ref{d:left-dual}.

\begin{prop}    \label{p:rigidity}
Let $\cM$ be an r-category and $\eps :A\tens B\rar{}\e$ a morphism in $\cM$. Then $(B,\eps )$ is
a right rigid dual of $A$ if and only if
\begin{enumerate}[$($a$)$]
\item $\eps$ induces an isomorphism $B\rar{\simeq} D^{-1}A$; and
 \sbr
\item the canonical morphism $B\tens A\rar{} B\odot A$ defined in \S\ref{sss:odot} is an isomorphism.
\end{enumerate}
In this case the canonical morphisms $B\tens Y\rar{} B\odot Y$ and $Y\tens A\rar{} Y\odot A$
are isomorphisms for all $Y\in\cM$.
\end{prop}

See \S\ref{ss:proof_of_rigidity_prop} for the proof of the proposition.

\begin{cor}   \label{c:rigidity1}
An object $X$ of an r-category $\cM$ is rigid if and only if the canonical morphisms
$X\tens DX\rar{}X\odot DX$ and $D^{-1}X\tens X\rar{}D^{-1}X\odot X$ defined in \S\ref{sss:odot}
are isomorphisms. Then the left rigid dual of $X$ equals $DX$, the right one equals $D^{-1}X$, and
the canonical morphisms $X\tens Y\rar{}X\odot Y$ and $Y\tens X\rar{}Y\odot X$ are isomorphisms for all
$Y\in\cM$.     \hfill\qedsymbol
\end{cor}

\begin{cor}   \label{c:rigidity2}
The following properties of an r-category $\cM$ are equivalent:
\begin{enumerate}[$($i$)$]
\item $\cM$ is rigid;
 \sbr
\item the canonical morphism $X\tens Y\rar{} X\odot Y$ defined in \S\ref{sss:odot} is an isomorphism
for every $X,Y\in\cM$;
 \sbr
\item the canonical morphism $X\tens DX\rar{}X\odot DX$ is an isomorphism for every $X\in\cM$.
\hfill\qedsymbol
\end{enumerate}
\end{cor}

\begin{rem}
Corollary \ref{c:rigidity2} is probably well known; for instance, the equivalence $(i)\iff(ii)$ is proved in the last paragraph of \cite[Section 5]{day-street}.
\end{rem}

\begin{cor}   \label{c:properness}
The monoidal category $\sD (G)$ from Example~\ref{example:dual-gen4}
is rigid if and only if $G$ is proper. The same is true for $\sD_G(G)$.
\end{cor}

\begin{proof}
Let $\cM$ be either $\sD (G)$ or $\sD_G(G)$.
We use the equivalence (i)$\Leftrightarrow$(ii) from Corollary~\ref{c:rigidity2}.
Recall that $\tens$ is convolution with compact support,
$\odot$ is convolution without compact support, and the morphism
$f_{XY}:X\tens Y\rar{} X\odot Y$ is the usual one (see Example~\ref{example:forgetsupp}). So if $G$ is proper then $f_{XY}$ is an isomorphism
for all $X,Y\in\cM$. Conversely, if $f_{XY}$ is an isomorphism for $X=Y=\ql$ (the constant sheaf on $G$), then $G$ is proper.
\end{proof}

\section{$D^2$ as a monoidal equivalence}  \label{s:D^2monoidal}
By \eqref{e:duality3}, for each $X,Y_1,Y_2\in\cM$ one has a canonical isomorphism
\begin{equation}  \label{e:D^21}
\Hom (X\otimes Y_1\otimes Y_2, K )\rar{\simeq}\Hom (D^2(Y_1\otimes Y_2)\otimes X, K ).
\end{equation}
On the other hand, writing $X\otimes Y_1\otimes Y_2$ as $(X\otimes Y_1)\otimes Y_2$ and applying
\eqref{e:duality3} twice one gets an isomorphism
\begin{equation}   \label{e:D^22}
\Hom (X\otimes Y_1\otimes Y_2, K )\rar{\simeq}\Hom (D^2Y_1\otimes D^2Y_2\otimes X, K ).
\end{equation}
Combining \eqref{e:D^21} and \eqref{e:D^22} one gets a functorial isomorphism
\begin{equation} \label{e:D^23}
\Hom (D^2(Y_1\otimes Y_2)\otimes X, K )\rar{\simeq}\Hom (D^2Y_1\otimes D^2Y_2\otimes X, K ),
\quad X,Y_1,Y_2\in\cM .
\end{equation}

\begin{lem}   \label{l:var_Yoneda}
Let $(\cM , K)$ be a \GV category and $Z_1,Z_2\in\cM$. Then every morphism
$\Hom (Z_1\tens X,K)\to\Hom (Z_2\tens X,K)$ functorial in $X\in\cM$ comes from
a unique morphism $Z_2\rar{}Z_1$.
\end{lem}

\begin{proof}
Use the isomorphism $\Hom (Z_i\tens X,K)\rar{\simeq}\Hom (Z_i,DX)$ and Yoneda's lemma.
\end{proof}

Lemma~\ref{l:var_Yoneda} shows
that the isomorphism~\eqref{e:D^23} comes from a unique functorial isomorphism
\begin{equation}   \label{e:D^24}
u_{Y_1,Y_2}: D^2(Y_1\otimes Y_2)\rar{\simeq} D^2Y_1\otimes D^2Y_2 , \quad Y_1,Y_2\in\cM .
\end{equation}

\begin{prop} \label{p:DD'}
The isomorphism \eqref{e:D^24} defines a monoidal structure on the functor $D^2:\cM\iso\cM$. The corresponding isomorphism $\e\rar{\simeq}D^2(\e)$ is equal to \eqref{e:D^2(1)}.
\end{prop}

As explained to us by J.~Ayoub, this can be checked directly
(by Lemma~\ref{l:var_Yoneda}, to prove that the two isomorphisms
$D^2(Y_1\otimes Y_2\otimes Y_3)\rar{\simeq} D^2Y_1\otimes D^2Y_2\otimes D^2Y_3$
are equal, it suffices to show that the corresponding isomorphisms
\[
\Hom (D^2(Y_1\otimes Y_2\otimes Y_3)\otimes X, K )\rar{\simeq}
\Hom (D^2Y_1\otimes D^2Y_2\otimes D^3Y_3\otimes X, K )
\]
are equal).
In \S\ref{ss:proof_D^2monoidal} we give a slightly different proof of Proposition~\ref{p:DD'}.

\section{Pivotal structures on Grothendieck-Verdier categories}\label{s:pivotal}

\subsection{The definition of pivotal structure}   \label{ss:pivotal}

\begin{defin}  \label{d:pivotal1}
A \emph{pivotal structure} on a Grothendieck-Verdier category $(\cM,K)$ is a functorial isomorphism
\begin{equation}   \label{e:psi1}
\psi_{X,Y}:\Hom (X\otimes Y,K )\rar{\simeq}\Hom (Y\otimes X,K ), \quad X,Y\in\cM
\end{equation}
such that
\begin{equation}   \label{e:pretty2}
\psi_{X\tens Y,Z}\circ \psi_{Y\tens Z,X}\circ\psi_{Z\tens X,Y}=\id , \quad X,Y,Z\in\cM ;
\end{equation}
\begin{equation}   \label{e:pretty1}
\psi_{X,Y}\circ \psi_{Y,X}=\id , \quad X,Y\in\cM .
\end{equation}
\end{defin}

In particular, one has a notion of pivotal structure on an r-category
(which can be considered as a Grothendieck-Verdier category with $K=\e$).

\begin{defin} \label{d:pivotal2}
A \emph{pivotal Grothendieck-Verdier category}
is a Grothendieck-Verdier category with a pivotal structure. A \emph{pivotal r-category}
is an r-category with a pivotal structure.
\end{defin}

The name ``pivotal category'' goes back to \cite[Definition 1.3]{FY}.

\begin{lem}   \label{l:pivotal}
Let $\cM$ be a Grothendieck-Verdier category and $\psi$ an isomorphism \eqref{e:psi1}
satisfying \eqref{e:pretty2}. Then $\psi$ satisfies \eqref{e:pretty1} if and only if $\psi_{K,\e}=\id$.
\end{lem}

\begin{proof}
Setting $Z=\e$ in \eqref{e:pretty2} we see that \eqref{e:pretty1} holds  if and only if the isomorphism
$\psi_{X,\e}:\Hom (X,K )\to\Hom (X,K )$ equals the identity for all $X$. By Yoneda's lemma, this happens if and only if  $\psi_{K,\e}=\id$.
\end{proof}

\begin{cor}   \label{c:pivotal}
If $\cM$ is an r-category then \eqref{e:pretty2} implies \eqref{e:pretty1}.
\end{cor}

\begin{rem}   \label{r:pivotal1}
By \eqref{e:pretty2} and \eqref{e:pretty1}, a pivotal structure on a Grothendieck-Verdier category defines for any integers $n\ge m\ge 1$
a {\em canonical\,} isomorphism
\[
\Hom (X_1\tens\ldots\tens X_n,K)\rar{\simeq}\Hom (X_m\tens\ldots\tens X_n\tens X_1\tens\ldots \tens
X_{m-1},K),\;  X_i\in\cM.
\]
\end{rem}

\subsection{Pivotal structures and isomorphisms $\Id\rar{\simeq}D^2$}
\begin{lem}[Cf.~\cite{egger-mccurdy}]        \label{l:pivotal2}
There is a one-to-one correspondence between functorial isomorphisms
\[
\psi_{X,Y}:\Hom (X\otimes Y,K )\rar{\simeq}\Hom (Y\otimes X,K ), \quad X,Y\in\cM
\]
and isomorphisms of functors $f:\Id_{\cM}\rar{\simeq} D^2$. Namely,
$\psi$ corresponds to $f$ if the diagram
\begin{equation}    \label{e:fpsi}
\xymatrix{
  \Hom(D^2Y\tens X,K)
  \ar[rr]^{(f_Y\tens\id_X)^*} & & \Hom(Y\tens X,K) \\
   & \Hom(X\tens Y,K) \ar[ur]^{\simeq}_{\psi_{X,Y}} \ar[ul]_{\simeq}^{\eqref{e:duality3}} &
   }
\end{equation}
commutes for all $X,Y\in\cM$. Here the left diagonal arrow is the isomorphism \eqref{e:duality3}.
\end{lem}

\begin{proof}
Use Lemma~\ref{l:var_Yoneda}.
\end{proof}

\begin{prop}    \label{p:2pivotal}
An isomorphism $f:\Id_{\cM}\rar{\simeq} D^2$ corresponds $($in the sense of Lemma \ref{l:pivotal2}$)$
to a pivotal structure if and only if it satisfies the following conditions:
 \sbr
\begin{enumerate}[$($i$)$]
\item $f$ is monoidal; and
 \sbr
\item $f_K:K\rar{\simeq} D^2K$ equals the isomorphism \eqref{e:D^2K}.
\end{enumerate}
In this case
\begin{equation}     \label{e:Df1}
f_{DX}=D(f_X)^{-1}, \quad\quad \forall X\in\cM.
\end{equation}
\end{prop}

The proof is given in \S\ref{s:pivotalproof}.

\begin{rems}        \label{r:pivotal3}
\begin{enumerate}[(i)]
\item If $\cM$ is an r-category then condition (ii) from Proposition~\ref{p:2pivotal} clearly
follows from condition (i). For more general \GV categories this is not always the case. For instance, consider
the pre-additive category $\cM$ with objects $0,\e, K$ and with
\[
\Hom (\e, K)=\Hom (K,\e )=0, \quad \End\e =\End K=A,
\]
where $A$ is a commutative unital ring. Define the tensor product $\cM\tens\cM\to\cM$ on objects
so that $K\tens K=0$ and $\e\tens X=X\tens\e=X$ for all $X\in\cM$, define it on morphisms using the
product in $A$, and take the associativity constraint in $\cM$ to be trivial. Then $\cM$ is a
Grothendieck-Verdier category. In this situation monoidal isomorphisms  $\Id\rar{\simeq} D^2$ bijectively correspond to
elements of $A^{\times}$, and only one of them defines a pivotal structure (namely, the isomorphism corresponding to $1\in A^{\times})$.
 \sbr
\item By the previous remark, in the case of r-categories a pivotal structure can
equivalently be defined to be a monoidal isomorphism $f:\Id\rar{\simeq} D^2$. It is this
definition that was used in works on rigid monoidal categories
(e.g., see \cite[Definition 2.7]{ENO}).
 \sbr
\item Here is a way to combine the two conditions on $f$ from Proposition~\ref{p:2pivotal}
into one. Let $\fA$ be the 2-groupoid of pairs consisting of a monoidal category and an object in it. A
Grothendieck-Verdier category $(\cM ,K)$ is an object in $\fA$. The monoidal structure on $D^2$
and the isomorphism $K\rar{\simeq} D^2(K)$ defined in Remark~\ref{r:r}(iv) allow us to
consider $D^2$ as a 1-automorphism of $(\cM ,K)\in\fA$. The two conditions on $f$ from
Proposition~\ref{p:2pivotal}  mean that $f:\Id\rar{\simeq} D^2$ is a 2-isomorphism in $\fA$.
\end{enumerate}
\end{rems}

\subsection{Examples of pivotal \GV categories}
\begin{example}
Every symmetric \GV category has an obvious pivotal structure.
\end{example}

\begin{example}   \label{example:pivotal}
The categories $\sD(G)$ and $\sD_G(G)$ from Example~\ref{example:dual-gen4} have a canonical pivotal structure (see \cite[\S{}A.2.3]{foundations}). The corresponding isomorphism $\Id\rar{\simeq} D^2=(\bD_G\circ\iota^*)^2$ comes from the obvious isomorphisms $(\bD_G)^2\rar{\simeq}\Id$, $(\iota^*)^2\rar{\simeq}\Id$,
$\bD_G\circ\iota^*\rar{\simeq}\iota^*\circ\bD_G$.
\end{example}

\begin{example}
Quite similarly to the previous example, one defines a canonical pivotal structure
on the \GV category $\sD(\Ga )$ from Example~\ref{example:dual-gen5}.
\end{example}

\section{Braided Grothendieck-Verdier categories}  \label{s:braidedGV}
A \emph{braided \GV category} is a Grothendieck-Verdier category $(\cM,K)$ equipped with
a braiding $\be_{X,Y}:X\tens Y\rar{\simeq}Y\tens X$.
For any \GV category $(\cM,K)$ the functor $D^2:\cM\iso\cM$ has a canonical monoidal structure,
see \S\ref{s:D^2monoidal}. The main goal of this section is to prove the following proposition.
\begin{prop}  \label{p:D2D4}
Let $(\cM,K,\be)$ be a braided Grothendieck-Verdier category. Then
\begin{enumerate}[$($i$)$]
\item
the monoidal functor $D^2:\cM\rar{\sim}\cM$ is braided;
 \sbr
 \item
there is a canonical monoidal isomorphism\footnote{Recall that a braided isomorphism between braided functors is the same as a monoidal isomorphism.} $D^4\rar{\simeq}\Id_{\cM}$.
\end{enumerate}
\end{prop}

To prove Proposition~\ref{p:D2D4}, we will
construct a monoidal equivalence between each of the monoidal functors $D^{\pm 2}$ and
a certain braided equivalence $J_{\cM}:\cM\iso\cM$, which was defined by Joyal and Street for \emph{any} braided category $\cM$.\footnote{In terms of \S\ref{ss:Lurie}, $J_{\cM}$ comes from the action of $SO(2)$ on the operad $\cE_2$.} The definition of
$J_{\cM}$ is recalled in \S\ref{ss:JS}, and the canonical monoidal isomorphisms
$\vt^\pm:J^{\pm 1}_{\cM}\rar{\simeq} D^2$ are constructed in \S\ref{sss:vartheta-plus-minus}. Using
$\vt^\pm$ we define in \S\ref{sss:D^4} a canonical monoidal isomorphism
$\ga_{\cM}:\Id_{\cM}\rar{\simeq} D^4$ and a certain monoidal isomorphism
$C_{\cM}:\Id_{\cM}\rar{\simeq} J_{\cM}^2$, which we call \emph{the canonical double-twist.} In fact,
$J_{\cM}$ is just the identity functor equipped with a nontrivial
monoidal structure, so one can consider $C_{\cM}$ as a (non-monoidal) automorphism of
$\Id_{\cM}$.\footnote{This automorphism is the image of the generator of $\pi_1(\Omega S^2)=\pi_2(S^2)$ under the action of $\Omega S^2$ on $\cM$ mentioned in \S\ref{ss:Lurie}.
Note that since our $\cM$ is a usual category rather than an $(\infty ,1)$-category this action
does not feel $\pi_i(\Omega S^2)=\pi_{i+1}(S^2)$ for $i>1$.}

\subsection{The Joyal-Street equivalence, twists, and double-twists.}  \label{ss:JS}

\begin{defin} \label{d:JS}
Let $\cM$ be a braided category. The \emph{Joyal-Street} equivalence
is the following braided equivalence $J=J_{\cM}:\cM\rar{\simeq}\cM$: as a functor,
$J_{\cM}=\Id_{\cM}$, but the isomorphism $J_{\cM}(X\tens Y)\rar{\simeq} J_{\cM}(X)\tens J_{\cM}(Y)$ equals
\begin{equation}  \label{e:square_of_braiding}
\be_{Y,X}\circ\be_{X,Y}: X\tens Y\rar{\simeq} X\tens Y.
\end{equation}
\end{defin}

\begin{lem} \label{l:JS-monoidal-structure}
The isomorphism \eqref{e:square_of_braiding} indeed defines a braided structure on the
identity functor $\Id_{\cM}:\cM\rar{}\cM$.
\end{lem}

We learned this lemma and its proof given below from \cite[Remark 6.1]{joyal-street}.

\begin{proof}
For $X,Y\in\cM$ we write $\be^+_{X,Y}=\be_{X,Y}:X\tens Y\rar{\simeq}Y\tens X$ and
$\be^-_{X,Y}=\be_{Y,X}^{-1}:X\tens Y\rar{\simeq}Y\tens X$.
Let $\cM^{opp}$ be the monoidal category opposite to $\cM$; thus $\cM^{opp}$ equals $\cM$ as a category, but the monoidal structure is given by $X\overset{opp}{\tens}Y=Y\tens X$. Then $J_{\cM}$ is equal to $\Phi^+(\Phi^-)^{-1}$, where $\Phi^\pm:\cM^{opp}\rar{\sim}\cM$ are the following monoidal equivalences: as a functor, $\Phi^\pm$ equals $\Id_{\cM}$, and the isomorphism from $\Phi^{\pm}(X)\tens\Phi^{\pm}(Y)=X\tens Y$ to $\Phi^{\pm}(X\overset{opp}{\tens}Y)=Y\tens X$ equals $\be^{\pm}_{X,Y}$. (The fact that $\Phi^\pm$ are indeed monoidal functors follows immediately from the hexagon axioms.) Moreover, if we equip $\cM^{opp}$ with the braiding
\[
X\overset{opp}{\tens}Y=Y\tens X \xrar{\ \ \be_{Y,X}\ \ } X\tens Y = Y \overset{opp}{\tens} X
\]
then $\Phi^\pm$ become braided monoidal functors (here the verification is a tautology), which implies that $J_{\cM}$ is also braided.
\end{proof}

\begin{rem}    \label{r:J-inverse}
$J_{\cM}^{-1}$ is the functor $\Id_{\cM}:\cM\rar{}\cM$ equipped with the braided structure
\begin{equation}  \label{e:inverse_JS}
(\be_{Y,X}\circ\be_{X,Y})^{-1}: X\tens Y\rar{\simeq} X\tens Y.
\end{equation}
In other words, $J_{\cM}^{-1}$ is the Joyal-Street equivalence for $\cM$ equipped with
the opposite braiding $\be^-_{X,Y}:=\be_{Y,X}^{-1}$.
\end{rem}

\begin{defin}\label{d:twists}
A \emph{twist} on a braided
category $\cM$ is a monoidal isomorphism $\te :\Id_{\cM}\rar{\simeq}J_{\cM}$, where $J_{\cM}$ is the
Joyal-Street equivalence (see Definition~\ref{d:JS}). A \emph{double-twist} on $\cM$ is a monoidal isomorphism $\Id_{\cM}\rar{\simeq}J_{\cM}^2$.
\end{defin}

\begin{rems}\label{r:twists}
\begin{enumerate}[$($i$)$]
\item It is easy to check that the above definition of twist is equivalent to the usual one, i.e., a
twist is an automorphism $\te$ of the identity functor on $\cM$ that satisfies
\[
\te_{X\tens Y} = \be_{Y,X}\circ\be_{X,Y}\circ(\te_X\tens\te_Y) \qquad \forall\,X,Y\in\cM.
\]
Similarly, a double-twist is an automorphism $f$ of the identity functor such that
\[
f_{X\tens Y} = (\be_{Y,X}\circ\be_{X,Y})^2\circ(f_X\tens f_Y) \qquad \forall\,X,Y\in\cM.
\]
 \sbr
\item The previous remark implies that for any twist $\te$ one has $\te_{\e}=\id_{\e}$
and for any double-twist  $f$ one has $f_{\e}=\id_{\e}$.
 \sbr
\item If $\te_1,\te_2:\Id_{\cM}\rar{\simeq}J_{\cM}$ are twists then
$\te_1\te_2:\Id_{\cM}\rar{\simeq}J^2_{\cM}$ is a double-twist.
Moreover, $\te_1\te_2=\te_2\te_1$. Indeed, $J_{\cM}$ equals $\Id_{\cM}$ as a functor,
so for each $X\in\cM$ the isomorphism $(\te_i)_X$ belongs to the center of $\Aut X$ and
$(\te_1\te_2)_X=(\te_1)_X\circ (\te_2)_X =(\te_2)_X\circ (\te_1)_X=(\te_2\te_1)_X\,$.
 \sbr
\item The set of all twists is either empty or a torsor over $\Aut^{\tens}(\Id_{\cM})$, i.e., the group of monoidal automorphisms of $\Id_{\cM}$. The same is true for double-twists.
The map $(\te_1,\te_2)\mapsto \te_1\te_2$ from Remark (iii) agrees with the action of
$\Aut^{\tens}(\Id_{\cM})$.
\end{enumerate}
\end{rems}

\subsection{The canonical monoidal isomorphisms
$J_{\cM}\rar{\simeq} D^2\lar{\simeq} J^{-1}_{\cM}$} \label{sss:vartheta-plus-minus} Let $(\cM,K,\be)$ be a braided \GV category. As in the proof of Lemma \ref{l:JS-monoidal-structure}, we write $\be^+_{X,Y}=\be_{X,Y}$ and $\be^-_{X,Y}=\be^{-1}_{Y,X}$ for all $X,Y\in\cM$.

\begin{defin}\label{d:vartheta}
For each $Y\in\cM$, we let $\vt^\pm_Y:Y\rar{\simeq}D^2Y$ be the unique isomorphism\footnote{The existence and uniqueness of $\vt^\pm_Y$ follows from Lemma \ref{l:var_Yoneda}} such that for every $X\in\cM$, the induced map
\[
\Hom(D^2 Y\tens X,K) \rar{} \Hom(Y\tens X,K)
\]
is equal to the composition
\[
\Hom(D^2 Y\tens X,K)
\rar{\simeq}           \Hom(X\tens Y,K) \xrar{\ \ (\be^\pm_{Y,X})^*\ \ } \Hom(Y\tens X,K),
\]
where the first arrow is inverse to the isomorphism \eqref{e:duality3}.

\end{defin}
Clearly $\vt^\pm_Y$ is functorial in $Y$, so we have isomorphisms of functors
$\vt^\pm :\Id_{\cM}\rar{\simeq}  D^2$. The next lemma may be considered as an equivalent
definition of $\vt^\pm$.

\begin{lem}    \label{l:braided-GV-category-phi-pm}
Let $(\cM,K,\be)$ be a braided Grothendieck-Verdier category, and let
$\vp^\pm:D^{-1}\rar{\simeq}D$ be the isomorphisms induced by the compositions
\[
\Hom(Y,D^{-1}X) \rar{\simeq} \Hom(X\tens Y,K) \xrar{\ \ (\be^\pm_{Y,X})^*\ \ } \Hom(Y\tens X,K) \rar{\simeq} \Hom(Y,DX)
\]
for all $X,Y\in\cM$. Then
\begin{equation}\label{e:vartheta-equivalent}
\vt^\pm_Y = \vp^\pm_{DY} \quad\text{for all } Y\in\cM
\end{equation}
and
\begin{equation}   \label{e:myQuestion2}
D(\vp^\pm_X) = (\vp^\mp_{DX})^{-1} \quad\text{for all } X\in\cM.
\end{equation}
\end{lem}

The lemma will be proved in \S\ref{ss:Mitya's_lemma}.

\begin{rem}
In view of Lemma \ref{l:braided-GV-category-phi-pm}, we have
\begin{equation}  \label{e:vt+&vt-}
\vt^\pm_{DX} = D(\vt^\mp_X)^{-1} \quad\text{for all } X\in\cM.
\end{equation}
\end{rem}

By Proposition \ref{p:DD'}, the functor $D^2:\cM\rar{\simeq}\cM$ is equipped with a canonical monoidal structure. On the other hand, we have the Joyal-Street monoidal equivalence $J_{\cM}:\cM\rar{\simeq}\cM$, see Definition~\ref{d:JS}.
Since $J_{\cM}$ equals $\Id_{\cM}$ as a functor, we can view $\vt^\pm$ as isomorphisms of functors $\vt^\pm:J^{\pm 1}_{\cM}\rar{\simeq}D^2$.

\mbr

The next result is proved in \S\ref{s:proof-p:monoidal-iso-id-DD}.

\begin{prop}\label{p:monoidal-iso-id-DD}
The isomorphisms $\vt^\pm:J^{\pm 1}_{\cM}\rar{\simeq}D^2$ are monoidal.
\end{prop}

Clearly Proposition~\ref{p:D2D4} follows from
Proposition~\ref{p:monoidal-iso-id-DD}.

\subsection{The canonical monoidal isomorphism $\Id_{\cM}\rar{\simeq} D^4$ and the canonical double-twist}        \label{sss:D^4}

Let $(\cM,K,\be)$ be a braided \GV category. In \S\ref{sss:vartheta-plus-minus} we defined monoidal isomorphisms $\vt^+:J_{\cM}\rar{\simeq}D^2$ and $\vt^-:J^{- 1}_{\cM}\rar{\simeq}D^2$.

\begin{defin}   \label{d:D^4}
We put $\ga_{\cM}=\vt^+\vt^-:\Id_{\cM}\rar{\simeq} D^4$ and call it the \emph{canonical monoidal isomorphism} between $\Id_{\cM}$ and $D^4$.
\end{defin}

The next two remarks give alternative formulas for $\ga_{\cM}$.

\begin{rem}
One has $\ga_{\cM}=\vt^-\vt^+$. To see this, note that if $\cC$ is any monoidal category and $c\in\cC$ is isomorphic to $\e_{\cC}$ then for each $f^+,f^-\in\Hom (\e_{\cC}, c)$ the morphism
$f^+\tens f^-:\e_{\cC}\rar{} c\tens c$ equals $f^-\tens f^+$. Now let $\cC$ be the monoidal category of
functors $\cM\rar{}\cM$, $c:=D^2$, $f^{\pm}:=\vt^{\pm}$ (recall that $J_{\cM}$ equals $\Id_{\cM}$
as a functor).
\end{rem}

\begin{rem}
Clearly $\vt^-$ defines a monoidal isomorphism $(\vt^-)^{(-1)}:D^{-2}\rar{\simeq}J_{\cM}$.
One can check that $\ga_{\cM}$ is equal to the isomorphism
$\Id_{\cM}\rar{\simeq} D^4$
corresponding to the composition $\vt^+\circ (\vt^-)^{(-1)}:D^{-2}\rar{\simeq} D^2$.
We do not use this fact in this article.
\end{rem}

On the other hand, the isomorphism
$(\vt^+)^{-1}\circ\vt^- :J^{-1}_{\cM}\rar{\simeq} J_{\cM}$ defines a monoidal isomorphism
$C_{\cM}:\Id_{\cM} \rar{\simeq} J_{\cM}^2$, i.e., a double-twist in the sense of
Definition~\ref{d:twists}.

\begin{defin}   \label{d:canonical_double-twist}
$C_{\cM}$ is called the \emph{canonical double-twist} of $(\cM ,K,\be)$.
\end{defin}

For each $X\in\cM$, the isomorphisms $\ga_{\cM}:\Id_{\cM}\rar{\simeq} D^4$ and
$C_{\cM}:\Id_{\cM} \rar{\simeq} J_{\cM}^2$ define isomorphisms
$\ga_X :X \rar{\simeq} D^4X$ and $C_X :X \rar{\simeq} X$ (recall that $J_{\cM}$ equals
$\Id_{\cM}$ as a functor). By definition,
\begin{equation}  \label{e:CX}
C_X=(\vt^+_X)^{-1} \circ \vt^-_X \, .
\end{equation}
\begin{lem}    \label{l:gaX}
$\ga_X=\vt^+_{D^2X} \circ \vt^-_X=D^2(\vt^+_X) \circ \vt^-_X=\vt^-_{D^2X} \circ \vt^+_X =D^2(\vt^-_X) \circ \vt^+_X\,$.
\end{lem}

\begin{proof}
$\vt^{\pm}:\Id_{\cM}\rar{\simeq} D^2$ is an isomorphism of functors, so for any $X,Y\in\cM$ and any
$f:X\rar{} Y$ one has $D^2(f)\circ\vt_X^{\pm}=\vt_Y^{\pm}\circ f$. Taking  $Y=D^2X$,
$f=\vt^{\pm}_X$ one gets $D^2(\vt^+_X)=\vt^+_{D^2X}\,$. Taking $Y=D^2X$,
$f=\vt^{\mp}_X$ one gets $D^2(\vt^{\mp}_X) \circ \vt^{\pm}_X=\vt^{\pm}_{D^2X} \circ \vt^{\mp}_X$.
\end{proof}

\bbr

\begin{rems}
\begin{enumerate}[$($i$)$]
\item Combining formula \eqref{e:CX} and Lemma~\ref{l:gaX}  with formula~\eqref{e:vt+&vt-} one
sees that
\begin{equation}   \label{e:D-conjugatingC}
C_{DX} = D(C_X)
\end{equation}
and
\begin{equation}   \label{e:D-conjugating_ga}
\ga_{DX} = D(\ga_X)^{-1}.
\end{equation}
 \sbr
\item By Remark~\ref{r:twists}(ii), one has $C_{\e}=\id_{\e}$. By \eqref{e:D-conjugatingC},
this implies that
\begin{equation}  \label{e:C_K}
C_K=\id_K   .
\end{equation}
\end{enumerate}
\end{rems}

\section{Pivotal structures on braided Grothendieck-Verdier categories}\label{s:pivotal-braided}

This section is closely related to \cite[Section 4]{egger-mccurdy}. We thank the referee for informing us about this fact.

\subsection{Pivotal structures and twists} The notion of a pivotal structure on a (not necessarily braided) Grothendieck-Verdier category was introduced in Definition \ref{d:pivotal1}. Recall that by Proposition \ref{p:2pivotal},
a pivotal structure on a \GV category $(\cM,K)$ is the same as
a monoidal isomorphism $f:\Id_{\cM}\rar{\simeq}D^2$ such that $f_K:K\rar{\simeq}D^2 K$ is equal to the isomorphism \eqref{e:D^2K}.  So by abuse of language, we often say that $f$ is a pivotal structure. Now suppose that $\cM$ is equipped with
a braiding $\be$.

\mbr

\begin{prop}\label{p:pivotal-twists}
Let $(\cM,K,\be)$ be a braided \GV category.
Then the map $f\mapsto (\vt^+)^{-1}\circ f$ defines a bijection between the set of pivotal structures $f:\Id_{\cM}\rar{\simeq}D^2$ and the set of twists $\te$ on $\cM$ that satisfy
$\te_K=\id_K$.
\end{prop}

\begin{proof}
 This follows immediately from Proposition \ref{p:monoidal-iso-id-DD} and Definition \ref{d:twists}.
\end{proof}

\begin{rem}
In \S\ref{ss:pivotal} we defined a pivotal structure to be an isomorphism
\[
\psi_{X,Y}: \Hom(X\tens Y,K)\rar{\simeq}\Hom(Y\tens X,K), \qquad X,Y\in\cM,
\]
satisfying certain properties.
It is easy to check that the relation between $\psi$ and the corresponding twist $\te$ is as
follows:
\[
\psi_{X,Y} = (\te_Y\tens\id_X)^*\circ\be_{Y,X}^* : \Hom(X\tens Y,K) \rar{\simeq} \Hom(Y\tens X,K) \rar{\simeq} \Hom(Y\tens X,K).
\]
\end{rem}

\subsection{The involution on the set of pivotal structures}
Let $(\cM,K,\be)$ be a braided Grothen\-dieck-Verdier category. In the next proposition we define a canonical involution on the set of all pivotal structures on $(\cM,K )$ or equivalently,
on the set of twists $\te$ on $(\cM ,\be )$ such that $\te_K=\id_K$.
The fixed points of this involution correspond to
{\em ribbon structures} (see Definition~\ref{d:ribbon} and Corollary~\ref{c:ribbon} below).

\begin{prop}  \label{p:involution}
Let $(\cM,K,\be)$ be a braided \GV category.
\begin{enumerate}[$($i$)$]
\item For every twist $\te :\Id_{\cM}\rar{\simeq} J_{\cM}$ there is a unique twist
$\te' :\Id_{\cM}\rar{\simeq} J_{\cM}$ such that $\te\te': \Id_{\cM}\rar{\simeq} J_{\cM}^2$ is equal
to the canonical double-twist $C_{\cM}$ from Definition~\ref{d:canonical_double-twist}.
 \sbr
\item The map $\te\mapsto\te'$ is an involution.
 \sbr
\item If $\te_K=\id_K$ then $\te'_K=\id_K$.
 \sbr
\item If $\te_K=\id_K$ then $\te'_X=D^{-1}(\te_{DX})$.
 \sbr
\item Suppose that $\te_K=\id_K$. Let $f:\Id_{\cM}\rar{\simeq}D^2$ and
$f':\Id_{\cM}\rar{\simeq}D^2$ be the pivotal structures corresponding to $\te$ and $\te'$
by Proposition~\ref{p:pivotal-twists}. Then the isomorphism $ff':\Id_{\cM}\rar{\simeq}D^4$
equals the canonical isomorphism $\ga_{\cM}:\Id_{\cM}\rar{\simeq}D^4$ from
Definition~\ref{d:D^4}.
\end{enumerate}
\end{prop}

\begin{rem}  \label{r:involution}
If $K\simeq\e$ then the condition $\te_K=\id_K$ holds automatically because by
Remark~\ref{r:twists}, $\te_{\e}=\id_{\e}$.
\end{rem}

\begin{proof}
Statements (i) and (ii) follow from Remarks~\ref{r:twists}(iii-iv).
Statement (iii) follows from formula \eqref{e:C_K}.

\mbr

Let us prove (iv).
By \eqref{e:CX} and the definition of $\te'$, this amounts to showing that

\begin{equation}  \label{e:to_be_proved}
\te_X\circ D^{-1}(\te_{DX})=(\vt^+_X)^{-1}\circ\vt^-_X, \qquad X\in\cM .
\end{equation}

By Proposition \ref{p:pivotal-twists}, $\te = (\vt^+)^{-1}\circ f$ for some pivotal structure
$f:\Id_{\cM}\rar{\simeq}D^2$. Then for every $X\in\cM$ one has $\te_X = (\vt^+_X)^{-1}\circ f_X$.
By
formula \eqref{e:Df1}, $D^{-1}(f_{DX})=f_X^{-1}$, so by
formula~\eqref{e:vt+&vt-},
$D^{-1}(\te_{DX})=f_X^{-1}\circ\vt_X^-$. Formula \eqref{e:to_be_proved} follows.

\mbr

Finally, we prove (v). By definition, $f=\vt^+\circ\te$ and
\[
f' = \vt^+\circ\te' = \vt^+\circ(\vt^+)^{-1}\circ\vt^-\circ\te^{-1} = \vt^-\circ\te^{-1}
\]
(in the last equality we view $\te^{-1}$ as an isomorphism $\Id_{\cM}\rar{\simeq}J_{\cM}^{-1}$). Since $\te$ belongs to the Bernstein center of $\cM$ (recall that $J_{\cM}=\Id_{\cM}$ as a functor), we have
\[
ff' = (\vt^+\circ\te)\cdot(\vt^-\circ\te^{-1}) = \vt^+\vt^- = \ga_{\cM}.
\]
\end{proof}

\section{Ribbon \GV categories} \label{s:ribbon}

\begin{defin}\label{d:ribbon}
A \emph{ribbon structure} on a braided \GV category $(\cM,K,\be)$ is a twist $\te$ on
$(\cM,\be)$ such that
\begin{equation}   \label{e:ribbon}
\te_X=D^{-1}(\te_{DX})\quad \mbox{ for all } X\in\cM.
\end{equation}
A \emph{ribbon \GV category} is a braided \GV category with a ribbon structure.
\end{defin}

\begin{lem}   \label{l:ribbon}
A twist $\te$ satisfies \eqref{e:ribbon} if and only if $\te_K=\id_K$ and $\te'=\te$, where
$\te'$ is defined in Proposition~\ref{p:involution}(i).
\end{lem}

\begin{proof}
By Proposition~\ref{p:involution}(iv), we only have to show that the equality $\te_K=\id_K$
follows from \eqref{e:ribbon}. This is clear because $K=D\e$ and
by Remark~\ref{r:twists}(ii), $\te_{\e}=\id_{\e}$.
\end{proof}

\begin{cor}   \label{c:ribbon}
The correspondence between twists and pivotal structures
$($see Proposition~\ref{p:pivotal-twists}$)$ induces a bijection between
ribbon structures on $(\cM,K,\be)$ and those pivotal structure $f:\Id_{\cM}\rar{\simeq}D^2$
that are invariant under the involution $f\mapsto f'$ from Proposition~\ref{p:involution}$($v$)$.
\end{cor}

\begin{proof}
This follows from Lemma~\ref{l:ribbon} and Propositions~\ref{p:pivotal-twists}-\ref{p:involution}.
\end{proof}

\begin{example}
The r-category $\sD_G(G)$ from Example~\ref{example:dual-gen4} has a canonical ribbon structure, see \cite[\S{}A.5]{foundations}. It corresponds (in the sense of Proposition \ref{p:pivotal-twists}) to the pivotal structure from Example~\ref{example:pivotal}.
If the group $G$ is finite and the ground field $k$ is algebraically closed
then $\sD_G(G)$ is the derived category of the abelian category $\cA$ formed
by modules over the quantum double of the group algebra of $G$, and the
above-mentioned ribbon structure on $\sD_G(G)$ comes from the standard ribbon structure
on $\cA$. (The definition of the quantum double and the standard ribbon structure
on $\cA$ can be found, e.g.,  in \cite[\S3.2]{BK}).
\end{example}

\begin{rem}
The identity \eqref{e:ribbon} holds if and only if for any $X,Y\in\cM$ and $B:X\tens Y\rar{} K$ one has
\begin{equation}   \label{e:without_D}
B\circ (\id_X\tens\te_Y)=B\circ (\te_X\tens\id_Y).
\end{equation}
Note that unlike \eqref{e:ribbon}, formula \eqref{e:without_D} makes sense in \emph{any}
braided category with a fixed object $K$ ($K$ does not have to be dualizing and $\cM$ does not
have to be Grothendieck-Verdier). We do not know whether condition \eqref{e:without_D} is really
interesting in this generality.
\end{rem}

\section{Relation between r-categories and \GV categories}   \label{s:relation}
In this section we will use the notions of idempotent arrow, closed idempotent, and Hecke subcategory (see Definitions~\ref{d:closed_idemp} and \ref{d:Hecke}).
Lemma~\ref{l:properties-triple} and Proposition~\ref{p:inverse-construction} below answer the following question: which \GV categories can be realized as Hecke subcategories of r-categories? In order to formulate the answer, we will need

\begin{lem}\label{l:idempotents-auxiliary}
Let $\cM$ be a monoidal category and $\pi:\e\rar{}e$ an idempotent arrow in $\cM$.
 \sbr
\begin{enumerate}[$($a$)$]
\item The isomorphisms $e=\e\tens e\xrar{\ \ \pi\tens\id_e\ \ }e\tens e$ and $e=e\tens\e\xrar{\ \ \id_e\tens\pi\ \ }e\tens e$ are equal.
 \sbr
\item If $u:e\tens e\rar{\simeq}e$ is the inverse of either of the two isomorphisms in $($a$)$, then the pair $(e,u)$ is a unit object \cite[Def.~2.1(3)]{foundations} of $e\cM e$.
 \sbr
\item If $\varpi:\e\rar{}e$ is any other idempotent arrow in $\cM$, there is a unique automorphism $f:e\rar{\simeq}e$ such that $f\circ\varpi=\pi$.
\end{enumerate}
\end{lem}

\begin{proof}
Part (a) is \cite[Lemma 2.10]{foundations}, part (b) follows from \cite[Lemma 2.18]{foundations}, and part (c) is \cite[Corollary 2.40]{foundations}.
\end{proof}

\subsection{Hecke subcategories of r-categories}\label{ss:hecke-subcategory-r-category}
Let $\cM$ be an r-category, and let $\pi:\e\rar{}e$ be an idempotent arrow in $\cM$. We have a canonical identification $D^2(\e)\rar{\simeq}\e$, so we may view $D^2(\pi)$ as a morphism $D^2(\pi):\e\rar{}D^2(e)$. Since $D^2$ has a natural monoidal structure, it follows that $D^2(\pi)$ is an idempotent arrow.

\mbr

Next suppose that $D^2(e)\cong e$. By Lemma \ref{l:idempotents-auxiliary}(c) there is a unique isomorphism $\vp:D^2(e)\rar{\simeq}e$ such that $\vp\circ D^2(\pi)=\pi$.
Moreover, the Hecke subcategory $\cM':=e\cM e$ is a \GV category with dualizing object $K':=De$, and the duality functor for $(\cM',K')$ can be canonically identified with the restriction $D\bigl\lvert_{\cM'}$, using the isomorphisms\footnote{See \cite[Lemma A.50]{foundations} for more details.}
\[
\Hom(X\tens Y,K') \rar{\simeq} \Hom(X\tens Y,\e) \rar{\simeq} \Hom(X,DY), \qquad X,Y\in\cM'.
\]
(where the first arrow is induced by $D\pi:K'\rar{}D\e=\e$).
We write $\e'=e$ for the unit object\footnote{Strictly speaking, we use the structure of a unit object on $e$ coming from Lemma \ref{l:idempotents-auxiliary}(b).} of $\cM'$, and we keep the notation $D$ for the duality functor of $(\cM',K')$.

\mbr

We define $f:K'\rar{}\e'$ to be the composition $K'=De\rar{D\pi}D\e=\e\rar{\pi}e=\e'$.

\subsection{Properties of the triple $(\cM',K',f)$} We begin with the following

\begin{rem}\label{r:dual-morphism-K-1}
If $(\cM,K)$ is any \GV category and $f:K\rar{}\e$ is a morphism in $\cM$, then using the canonical identifications $DK\rar{\simeq}\e$ and $D\e\rar{\simeq}K$, we can view $Df$ as a morphism $K\rar{}\e$.
\end{rem}

\begin{lem}\label{l:properties-triple}
Let $\e\rar{\pi}e$ be an idempotent arrow in an r-category $\cM$, and let $(\cM',K',f)$ be the corresponding triple constructed as in \S\ref{ss:hecke-subcategory-r-category}. Then
 \sbr
\begin{enumerate}[$($a$)$]
\item $Df=f$ $($cf.~Remark \ref{r:dual-morphism-K-1}$)$;
 \sbr
\item for each $X\in\cM'$, the map $g\mapsto g\circ\pi$ is a bijection $\Hom(\e',X)\rar{\simeq}\Hom(\e,X)$;
 \sbr
\item for each $X\in\cM'$, the map $h\mapsto D\pi\circ h$ is a bijection $\Hom(X,K')\rar{\simeq}\Hom(X,\e)$.
\end{enumerate}
\end{lem}

\begin{proof}
(a) With our identifications, we have $D^2(\pi)=\pi$. Therefore $Df=D(\pi\circ D\pi)=D^2(\pi)\circ D\pi=\pi\circ D\pi=f$.

\mbr

(b) This follows from the fact \cite[Prop.~2.22(a)]{foundations} that the functor $Y\mapsto e\tens Y$ is left adjoint to the inclusion functor $\cM'=e\cM e\into\cM$.

\mbr

(c) This follows from (b) using the fact that $D$ is an anti-autoequivalence of $\cM$.
\end{proof}

\subsection{The inverse construction} \label{ss:inverse-construction}
\begin{prop}\label{p:inverse-construction}
Let $(\cM',K')$ be a \GV category with unit object $\e'$ and duality functor $D$, and let $f:K'\rar{}\e'$ be a morphism such that $Df=f$ $($cf.~Remark \ref{r:dual-morphism-K-1}$)$. Then the triple $(\cM',K',f)$ arises from a closed idempotent in an r-category by means of the construction described in \S\ref{ss:hecke-subcategory-r-category}.
\end{prop}

The proof of Proposition~\ref{p:inverse-construction} will be given in
\S\S\ref{ss:M-abstract}--\ref{ss:proof-p:inverse-construction}. In fact, given $(\cM',K')$ and $f$ we will construct there a concrete $r$-category $\cM$ and a closed idempotent $e\in\cM$ such that
$(\cM',K',f)$ arises from $(\cM ,e)$. One can characterize this pair $(\cM ,e)$ by a universal property, see Remark~\ref{r:univ_property}.

\begin{rem}
The assumption $Df=f$ in Proposition~\ref{p:inverse-construction} is
\emph{not} satisfied automatically, see \S\ref{ss:not_automatic}.
\end{rem}


\part{Proofs of the main results}

\section{Rigidity}\label{s:proof-rigidity}
In this section we prove Proposition~\ref{p:rigidity} and Proposition~\ref{p:Denis}.

\subsection{Recollections on rigid duals}  \label{ss:rigidity}
Let $\cM$ be  a monoidal category and $\eps :A\tens B\rar{}\e$ a morphism.
\begin{defin}\label{d:left-dual}
We say\footnote{Some authors use the opposite convention for ``left'' and ``right''. The advantage of our convention is that if $\cM$ is the category of endofunctors of some category then the left dual is the same as a left adjoint functor.} that
$(A,\eps )$ is a \emph{left rigid dual} of $B$ or that $(B,\eps )$ is a \emph{right rigid dual} of $A$ if
there exists $c:\e\rar{}B\tens A$ such that the compositions
\begin{equation}\label{e:left-dual-second-composition}
A = A\tens\e \xrar{\ \ \id_{A}\tens c \ \ } A\tens B\tens A \xrar{\ \ \eps\tens\id_{A}\ \ } \e\tens A=A
\end{equation}
and
\begin{equation}\label{e:left-dual-first-composition}
B = \e\tens B \xrar{\ \  c \tens\id_B\ \ } B\tens A\tens B \xrar{\ \ \id_B\tens\eps\ \ } B\tens\e = B
\end{equation}
are equal to $\id_A$ and $\id_{B}\,$, respectively.
An object of $\cM$ is said to be \emph{rigid} if it has a left rigid dual and a right one.
$\cM$ is said to be \emph{rigid} if each of its objects is.
\end{defin}

It is well known that the left or right rigid dual of an object $X\in\cM$ is unique up to unique isomorphism.
We denote the left rigid dual of $X$ by $X^*$ and the right one by ${}^*\!X$.
It is also known that in the situation of Definition~\ref{d:left-dual} the morphism $c:\e\rar{}B\tens A$
is unique. We will formulate a criterion for its existence, which goes back to
\cite{deligne-cat-tan,joyal-street}.

\mbr

If $\cM$ is a monoidal category and $X,Y\in\cM$ are objects such that the functor $Z\mapsto \Hom(X\tens Z,Y)$ is representable, then, following \cite{deligne-cat-tan}, we denote the representing object by $\HOM(X,Y)$.

\mbr

\begin{prop}\label{p:criterion-for-duality}
Let $\cM$ be a monoidal category and $\eps :A\tens B\rar{}\e$ a morphism in $\cM$. The following statements are equivalent:
 \sbr
\begin{enumerate}
\item[$($i$)$] $(B,\eps )$ is a right rigid dual of $A$
$($equivalently, $(A, \eps )$ is a left rigid dual of $B${}$)$;
 \sbr
\item[$($ii$)$] $\HOM(A,Y)$ exists for each $Y\in\cM$; moreover,
the morphism $B\tens Y\rar{}\HOM(A,Y)$ that comes from
$\eps\tens\id_Y:A\tens B\tens Y\rar{}\e\tens Y=Y$
is an isomorphism for every $Y\in\cM$;
 \sbr
\item[$($ii$')$] for all $Y,Z\in\cM$, the map
\begin{equation}\label{e:JS-map}
\Hom(Z,B\tens Y) \rar{} \Hom(A\tens Z,Y)
\end{equation}
that takes an element $f\in\Hom(Z,B\tens Y)$ to the composition
\[
A\tens Z \xrar{\ \ \id_A\tens f\ \ } A\tens B\tens Y \xrar{\ \ \eps\tens\id_Y\ \ } \e\tens Y = Y
\]
is bijective;
 \sbr
\item[$($iii$)$]
$\HOM(A,\e)$ and $\HOM(A,A)$ exist; in addition, the morphisms
$B\rar{}\HOM(A,\e)$ and $B\tens A\rar{}\HOM(A,A)$ defined in $($ii$)$ are isomorphisms;
 \sbr
\item[$($iii$')$] the map \eqref{e:JS-map} is bijective for $Y=\e$, $Z=B$ and for $Y=A$, $Z=\e$;
 \sbr
\item[$($iii$'')$] the map  \eqref{e:JS-map} is injective for $Y=\e$, $Z=B$ and surjective for
$Y=A$, $Z=\e$.
\end{enumerate}
\end{prop}

\begin{rems}
\begin{enumerate}[(1)]
\item It is easy to see that
$(ii)\Leftrightarrow (ii')$ and $(iii)\Rightarrow (iii')$.
Tautologically, $(ii)\Rightarrow(iii)$ and $(ii')\Rightarrow(iii')\Rightarrow(iii'')$.
 \sbr
\item The equivalence between $(ii)$ and $(i)$ is proved in \cite[Prop.~2.3]{deligne-cat-tan}.
The equivalence between $(ii')$ and $(i)$ is stated in \cite[p.~70]{joyal-street}. So it remains to
prove that $(iii'')\Rightarrow(i)$.
\end{enumerate}
\end{rems}

\begin{proof}[Proof of the implication $(iii'')\Rightarrow(i)$ in Proposition \ref{p:criterion-for-duality}] Applying hypothesis $(iii'')$ with $Y=A$ and $Z=\e$, we see that there is a morphism $c:\e\rar{}B\tens A$ such that the composition \eqref{e:left-dual-second-composition}
equals $\id_A$. Now let $\al$ denote the composition \eqref{e:left-dual-first-composition}.
It remains to show that $\al=\id_B$. Using the fact that the composition
\eqref{e:left-dual-second-composition} equals $\id_A$, it is easy to check that the composition
\[
A\tens B \xrar{ \id_A\tens\al } A\tens B \xrar{\ \ \eps\ \ } \e
\]
equals $\eps$. Thus the assumption of $(iii'')$ with $Y=\e$ and $Z=B$ forces $\al=\id_B$.
\end{proof}

\subsection{Proof of Proposition~\ref{p:rigidity}}   \label{ss:proof_of_rigidity_prop}
\begin{proof}
We will apply the equivalences (i)$\Leftrightarrow$(ii)$\Leftrightarrow$(iii) from
Proposition~\ref{p:criterion-for-duality}.

\mbr

First, recall that by Remark~\ref{r:r}(iii), the existence of
$\HOM (A,Y)$ is automatic; namely,
\[
\HOM (A,Y)=D^{-1}(DY\tens A),
\]
which can also be written as $\HOM (A,Y)=D^{-1}A\odot Y$ by formula \eqref{e:otherproduct}.
In particular, $\HOM (A,\e )=D^{-1}A$.

\mbr

So the condition that the morphism $B\rar{}\HOM(A,\e )$ is an isomorphism
(see Proposition~\ref{p:criterion-for-duality}) is equivalent to condition (a) of Proposition \ref{p:rigidity}. If it holds, the morphism $B\tens Y\rar{}\HOM(A,Y)$ from Proposition~\ref{p:criterion-for-duality}
can be considered as a morphism $B\tens Y\rar{}\HOM (DB,Y)=B\odot Y$, and one checks that it
equals the morphism $B\tens Y\rar{} B\odot Y$ defined in \S\ref{sss:odot}.

\mbr

Now applying the equivalence (i)$\Leftrightarrow$(iii) from
Proposition~\ref{p:criterion-for-duality} we see that $(B,\eps )$ is
a right rigid dual of $A$ if and only if (a) and (b) hold. Applying the equivalence (i)$\Leftrightarrow$(ii)
we see that in this case the canonical morphism $B\tens Y\rar{} B\odot Y$ is an isomorphism for every $Y\in\cM$. Now replacing $\tens :\cM\times\cM\rar{}\cM$ with the opposite tensor product and $B$ with $A$ we see that the canonical isomorphism $Y\tens A\rar{} Y\odot A$  is an isomorphism for every
$Y\in\cM$.
 \end{proof}

\subsection{Proof of Proposition~\ref{p:Denis}}  \label{sss:inv&rig}

\begin{lem}   \label{l:HomKK}
Let $(\cM ,K)$ be a Grothendieck-Verdier category. Then the canonical morphisms
$\e\rar{}\HOM (K,K)$ and $\e\rar{}\HOM' (K,K)$ are isomorphisms.
\end{lem}

Here $\HOM$ and $\HOM'$ are the internal $\Hom$'s, see Remark~\ref{r:r}(iii).

\begin{proof}
Use \eqref{e:internal_Hom2a}-\eqref{e:internal_Hom1b} and \eqref{e:D1}.
\end{proof}

Now let us prove Proposition~\ref{p:Denis}, which says that
a dualizing object of a monoidal category is invertible if and only if it is rigid.

\begin{proof}
Any invertible object is rigid. Now suppose that a dualizing object $K$ of a monoidal category $\cM$
is rigid. Let ${}^*\!K$ (resp. $K^*$) be its right (resp. left) rigid dual. By
Proposition~\ref{p:criterion-for-duality}(iii), ${}^*\!K\tens K\simeq\HOM (K,K)$, so
Lemma~\ref{l:HomKK} shows that ${}^*\!K\tens K\simeq\e$. Similarly, $K\tens K^*\simeq\e$.
\end{proof}

\section{The monoidal structure on $D^2$}  \label{s:proof_D^2monoidal}
\subsection{Proof of Proposition~\ref{p:DD'}} \label{ss:proof_D^2monoidal}
We first make an obvious remark, then formulate its
``categorification,'' and finally explain how to apply it to define
a monoidal structure on $D^2:\cM\rar{}\cM$, which, in fact, equals the one defined by
\eqref{e:D^24}.

\sbr

\subsubsection{Obvious remark} Let $A$ be an associative ring and let $N$ be
an $(A,A)$-bimodule. Suppose that for some $n_0\in N$ the maps $A\to
N$ defined by $a\mapsto n_0a$ and $a\mapsto an_0$ are  injections
with the same image. Define the map $\varphi :A\to A$ by the
equality $an_0=n_0\varphi (a)$. Then $\varphi$ is a ring
automorphism.

\sbr

\subsubsection{Categorification: a way to construct monoidal auto-equivalences}
\label{sss:Categorification}
Let $\cA$ be a monoidal category and let $\cN$
be an $(\cA ,\cA )$-bimodule category (i.e., we are given a monoidal functor from
$\cA\times\cA^{opp}$ to the monoidal category of functors
$\cN\to\cN$, where $\cA^{opp}$ is the category $\cA$ equipped with the opposite tensor product).
Suppose that for some $n_0\in\cN$ the functors
$\cA\to\cN$ defined by $X\mapsto n_0\otimes X$ and $X\mapsto X\tens
n_0$ are  fully faithful and have the same essential image. Then
there exists an equivalence $\Phi :\cA\iso\cA$ such that one has
isomorphisms $f_X:X\tens n_0\rar{\simeq} n_0\tens\Phi (X)$ functorial in
$X\in\cA$; such a pair $(\Phi ,f)$ is unique up to unique isomorphism.
We claim that $\Phi$ has a canonical structure of {\em monoidal\,}
equivalence. Namely, define $u_{X_1,X_2}:\Phi (X_1\tens X_2)\rar{\simeq}\Phi
(X_1)\tens\Phi (X_2)$ so that the diagram
\begin{equation}    \label{e:Phi_monoidal}
\xymatrix{
 X_1\tens X_2\tens n_0  \ar[d]_{\id_{X_1}\tens f_{X_2}} \ar[rr]^{f_{X_1\tens X_2}\ \ \ \ \ \ \ } & &
 n_0\tens \Phi (X_1\tens   X_2)\ar[d]^{\id_{n_0}\tens u_{X_1,X_2}} \\
 X_1\tens n_0\tens \Phi (X_2) \ar[rr]^{\ \ \ f_{X_1}\tens\id_{\Phi (X_2)}\ \ \ \ \ \ } & &
 n_0\tens \Phi (X_1)\tens\Phi (X_2)}
\end{equation}
commutes. Similarly, we have a natural isomorphism $\Phi(\e)\rar{\simeq}\e$. The isomorphisms $u_{X_1,X_2}$ are compatible with
the associativity constraint and the unit constraints and thus define a structure of
monoidal functor on $\Phi$.

\sbr

\subsubsection{Application}      \label{sss:Application}
Take $\cA =\cM^{\circ}$, where $\cM^{\circ}$ is the category dual to $\cM$.
Let $\cN$ be the
category of functors $F: \cM^{\circ}\to$\{Sets\}. It has a canonical
structure of $(\cM^{\circ},\cM^{\circ})$-bimodule category such that
$(m_1\tens F\tens m_2 )(m)=F(m_2\tens m\tens m_1)$ (the
associativity constraints of the bimodule category are the obvious ones).
Let $\cY:\cM\into\cN$ be the Yoneda embedding and $n_0:= \cY (K )\in\cN$.
Using \eqref{e:duality1} and \eqref{e:duality2} one checks that
$m\otimes n_0= \cY (Dm)$ and $n_0\otimes m= \cY (D^{-1}m)$. So the above
construction of a monoidal equivalence $\Phi :\cM^{\circ}\iso\cM^{\circ}$ is applicable and the functor
$\cM\rar{\sim}\cM$ corresponding to $\Phi$ equals $D^2$. Thus we get a
structure of a monoidal functor on $D^2$.

\sbr

\subsubsection{Conclusion}
One checks that the isomorphism
$u_{X_1,X_2}:D^2 (X_1\tens X_2)\rar{\simeq} D^2 (X_1)\tens D^2 (X_2)$ defined above equals
the isomorphism~\eqref{e:D^24} and that the isomorphism $\e\rar{\simeq}D^2(\e)$ defined above equals the isomorphism \eqref{e:D^2(1)}. Proposition~\ref{p:DD'} follows.

\subsection{A remark (to be used in  \S\ref{s:proof-p:monoidal-iso-id-DD})}
\label{ss:isomorphism_monoidal}
Suppose that in the situation of \S\ref{sss:Categorification} we have two pairs $(\Phi ,f)$
and $(\tilde\Phi ,\tilde f)$, so the functors $\Phi$ and $\tilde\Phi$ are both monoidal.
Let $\alpha :\tilde\Phi\rar{\simeq}\Phi$ be the unique isomorphism such that the diagram
\[
\xymatrix{
  n_0\tens \tilde\Phi (X)  \ar[rr]^{\id_{n_0}\tens\al_X} & & n_0\tens \Phi (X)  \\
   & X \tens  n_0 \ar[ul]^{\tilde f}_{\simeq} \ar[ur]^{\simeq}_ f &
   }
 \]
commutes. Then \emph{$\alpha$ is monoidal.}
To see this, compare \eqref{e:Phi_monoidal} with a similar commutative square for
$(\tilde\Phi ,\tilde f)$ by drawing a cube.

\section{Proof of Lemma~\ref{l:braided-GV-category-phi-pm}}
\subsection{An abstract lemma}
In this subsection $\cM$ is an \emph{abstract} category rather than a monoidal one.
The following lemma is used in the proof of Lemma ~\ref{l:braided-GV-category-phi-pm} and Proposition~\ref{p:2pivotal}.

\begin{lem}  \label{l:attempt2}
Let $\cM$ be a category equipped with an anti-equivalence $D:\cM\rar{}\cM$. Let $S$ be the set of functorial families of bijections
\[
\ga_{X,Y} : \Hom(X,DY) \rar{\simeq} \Hom(Y,DX), \qquad X,Y\in\cM.
\]
 \sbr
\begin{enumerate}[$($a$)$]
\item For each $\ga\in S$ there is a unique isomorphism $\vp^\ga:D^{-1}\rar{\simeq}D$ such that for all $X,Y\in\cM$ the corresponding map $\Hom(Y,D^{-1}X)\rar{\simeq}\Hom(Y,DX)$ is equal to the composition
    \begin{equation}\label{e:def-varphi-gamma}
    \Hom(Y,D^{-1}X) \rar{D} \Hom(X,DY) \xrar{\ \ \ga_{X,Y}\ \ } \Hom(Y,DX).
    \end{equation}
    The map $S\rar{}\Isom(D^{-1},D)$ given by $\ga\mapsto\vp^\ga$ is a bijection.
 \sbr
\item Define an involution $\vee:S\rar{\simeq}S$ by $(\ga^\vee)_{X,Y}=\ga_{Y,X}^{-1}$. Define a bijection \[\vee:\Isom(D^{-1},D)\rar{\simeq}\Isom(D^{-1},D)\] by $(\vp^\vee)_X:=D(\vp^{-1}_{D^{-1}X})$. Then
\begin{equation}\label{e:duality-key-1}
\vp^{\ga^\vee} = (\vp^\ga)^\vee \qquad\forall\,\ga\in S.
\end{equation}
In particular, $\vee$ is also an involution on $\Isom(D^{-1},D)$ and can alternatively be defined by the formula $(\vp^\vee)_X=D^{-1}(\vp_{DX}^{-1})$ for all $X\in\cM$.
\end{enumerate}
\end{lem}

\begin{proof}
Yoneda's lemma implies statement (a) and gives explicit formulas for $\vp^\ga$ and $(\vp^\ga)^{-1}$. Namely, to obtain a formula for $\vp^\ga_X$, where $X\in\cM$, apply the composition \eqref{e:def-varphi-gamma} for $Y=D^{-1}X$ to $\id_{D^{-1}X}\in\Hom(D^{-1}X,D^{-1}X)$; to obtain a formula for $(\vp^\ga_X)^{-1}$, consider the composition \eqref{e:def-varphi-gamma} for $Y=DX$ and apply its inverse to $\id_{DX}\in\Hom(DX,DX)$. Thus
\begin{equation}\label{e:varphi-gamma-X}
\vp^\ga_X = \ga_{X,D^{-1}X}(\id_X)
\end{equation}
and
\begin{equation}\label{e:varphi-gamma-X-inverse}
(\vp^\ga_X)^{-1} = D^{-1}\bigl( \ga_{X,DX}^{-1}(\id_{DX}) \bigr).
\end{equation}
Now let us deduce \eqref{e:duality-key-1} from \eqref{e:varphi-gamma-X}--\eqref{e:varphi-gamma-X-inverse}. By \eqref{e:varphi-gamma-X}, $\vp^{\ga^\vee}_X=\ga^{-1}_{D^{-1}X,X}(\id_{D^{-1}X})$. Comparing this with \eqref{e:varphi-gamma-X-inverse}, we see that $\vp^{\ga^\vee}_X = D\bigl( (\vp^\ga_{D^{-1}X})^{-1} \bigr)$, which is equivalent to \eqref{e:duality-key-1}. The remaining assertions of (b) follow at once.
\end{proof}

\begin{cor}     \label{c:abstract}
Let $\ga$ and $\varphi^{\ga}$ be as in Lemma~\ref{l:attempt2}. For brevity, set
$\varphi:=\varphi^{\ga}$. Then the following properties are equivalent:
\begin{enumerate}[$($i$)$]
\item $\ga_{X,Y}\circ\ga_{Y,X}=\id\,$ for all $X,Y\in\cM$.
 \sbr
\item  $\varphi_X=D(\varphi_{D^{-1}X})^{-1}$ for all $X\in\cM$;
\end{enumerate}
\end{cor}

\begin{proof}
Property~(i) means that $\ga^\vee=\ga$. Property~(ii) means that $\varphi^\vee=\varphi$. So
(i)$\Leftrightarrow$(ii) by Lemma~\ref{l:attempt2}(b).
\end{proof}

\subsection{Proof of Lemma~\ref{l:braided-GV-category-phi-pm}}     \label{ss:Mitya's_lemma}
Let $(\cM, K,\be )$ be a braided \GV category and $D:\cM\iso\cM$ the duality functor.
Let $\be^\pm$ be as in Lemma~\ref{l:braided-GV-category-phi-pm}.
Apply Lemma~\ref{l:attempt2} to the functorial families of isomorphisms
\[
\ga^\pm_{X,Y} : \Hom(X,DY) \rar{\simeq} \Hom(Y,DX), \qquad X,Y\in\cM,
\]
induced by the pullback maps
\[
(\be^\pm_{Y,X})^* : \Hom(X\tens Y,K) \rar{\simeq} \Hom(Y\tens X,K).
\]
The isomorphism $\vp^{\ga^\pm}$ from Lemma~\ref{l:attempt2}(a) equals the isomorphism $\vp^\pm$
from Lem\-ma~\ref{l:braided-GV-category-phi-pm}.
With the notation of Lemma~\ref{l:attempt2}(b), $(\ga^\pm)^\vee=\ga^\mp$. Thus the second assertion of Lemma~\ref{l:braided-GV-category-phi-pm} is equivalent to \eqref{e:duality-key-1}.

\mbr

To prove the first assertion, note that the composition
\[
\Hom(D^2 Y,DX) \xrar{\ \ D^{-1}\ \ } \Hom(X,DY) \xrar{\ \ \ga^\pm_{X,Y}\ \ } \Hom(Y,DX)
\]
equals $(\vt^\pm_Y)^*:\Hom(D^2Y,DX)\rar{\simeq}\Hom(Y,DX)$ by Definition \ref{d:vartheta}. Hence the map \[\Hom(Y,D^{-1}X)\rar{}\Hom(Y,DX)\] induced by $\vp^\pm_X:D^{-1}X\rar{}DX$ is equal to the composition
\[
\Hom(Y,D^{-1}X) \xrar{\ \ D^2\ \ } \Hom(D^2Y,DX) \xrar{\ \ (\vt^\pm_Y)^*\ \ } \Hom(Y,DX);
\]
more explicitly, $\vp^\pm_X\circ f=(D^2f)\circ\vt_Y^\pm$ for all $f\in\Hom(Y,D^{-1}X)$. Taking $X=DY$ and $f=\id_Y$ yields $\vt^\pm_Y=\vp^\pm_{DY}$, as claimed.

\section{Proof of Proposition~\ref{p:2pivotal}}   \label{s:pivotalproof}

Throughout this section we fix a \GV category $(\cM,K)$ together with a functorial family of isomorphisms
\[
\psi_{X,Y}:\Hom (X\otimes Y,K )\rar{\simeq}\Hom (Y\otimes X,K ), \qquad X,Y\in\cM,
\]
and let $f:\Id_{\cM}\rar{\simeq}D^2$ be the corresponding isomorphism (see Lemma~\ref{l:pivotal2}).

\subsection{Formulating the lemmas}
The following lemmas will be proved in
\S\ref{ss:proof-l:monoidality-of-f}--\ref{ss:proofpivot2}.

\begin{lem}\label{l:monoidality-of-f}
With the notation above, $f$ is monoidal if and only if
\begin{equation}\label{e:ugly}
\psi_{X\tens Y,Z} \circ \psi_{Y\tens Z,X} = \psi_{Y,Z\tens X} \qquad \forall\,X,Y,Z\in\cM.
\end{equation}
\end{lem}

\begin{lem}\label{l:equivalent-to-pretty1}
Identity \eqref{e:pretty1} holds if and only if $f_{DX}=D(f_X)^{-1}$ for all $X\in\cM$.
\end{lem}

\begin{lem}   \label{l:pivot2}
The following conditions are equivalent:
\begin{enumerate}[$($a$)$]
\item $f_K:K\rar{\simeq} D^2K$ equals the isomorphism \eqref{e:D^2K};
 \sbr
\item one has
\begin{equation}   \label{e:pretty6}
\psi_{\e,X}=\id \quad \mbox{ for all } X\in\cM ;
\end{equation}
 \sbr
\item $\psi_{\e,K}:\Hom (K,K)\to\Hom (K,K)$ maps $\id_K$ to itself.
\end{enumerate}
\end{lem}

\subsection{Deducing Proposition~\ref{p:2pivotal} from the lemmas.}
Assume that $\psi$ is a pivotal structure. Identities \eqref{e:pretty2}--\eqref{e:pretty1} imply
\eqref{e:ugly}, so $f$ is monoidal by Lemma \ref{l:monoidality-of-f}. Next,
Lemma \ref{l:equivalent-to-pretty1} shows that $f_{DX}=D(f_X)^{-1}$ for all $X\in\cM$.
Taking $X=\e$, we obtain $f_K=D(f_\e)^{-1}$. Since $f$ is monoidal, $f_\e$ is equal to the isomorphism \eqref{e:D^2(1)}, and therefore $f_K$ is equal to the isomorphism \eqref{e:D^2K}. Thus
$f$ satisfies conditions (i)--(ii) of Proposition \ref{p:2pivotal} and also satisfies \eqref{e:Df1}.

\mbr

Now suppose that $f$ satisfies conditions (i)--(ii) of Proposition \ref{p:2pivotal}. Then
\eqref{e:ugly} and \eqref{e:pretty6}
hold by Lemmas~\ref{l:monoidality-of-f} and \ref{l:pivot2}.
Setting $Y=\e$ in \eqref{e:ugly}
and using \eqref{e:pretty6} we see that
\begin{equation}   \label{e:pretty5}
\psi_{X,Z}\circ\psi_{Z,X}=\id .
\end{equation}
Clearly \eqref{e:ugly} and \eqref{e:pretty5} imply that $\psi$ is a pivotal structure,
see Definition~\ref{d:pivotal1}.  \qed

\subsection{Proof of Lemma \ref{l:monoidality-of-f}}\label{ss:proof-l:monoidality-of-f} Fix $X,Y,Z\in\cM$.
According to the definition of the correspondence between $f$ and $\psi$
(see Lemma~\ref{l:pivotal2}), the isomorphism
\[
\psi_{Y,Z\tens X} : \Hom(Y\tens Z\tens X,K) \rar{\simeq} \Hom(Z\tens X\tens Y,K)
\]
is equal to the composition of
\[
\Hom(Y\tens Z\tens X,K) \xrar{\ \eqref{e:duality3}\ } \Hom(D^2(Z\tens X)\tens Y,K)
\]
and
\[
(f_{Z\tens X}\tens\id_Y)^* : \Hom(D^2(Z\tens X)\tens Y,K) \rar{\simeq} \Hom(Z\tens X\tens Y,K).
\]
Similarly, the isomorphism
\[
\psi_{X\tens Y,Z} \circ \psi_{Y\tens Z,X} : \Hom(Y\tens Z\tens X,K) \rar{\simeq} \Hom(Z\tens X\tens Y,K)
\]
is equal to the composition of the isomorphisms
\[
\Hom(Y\tens Z\tens X,K) \xrar{\ \eqref{e:duality3}\ } \Hom(D^2X\tens Y\tens Z,K) \xrar{\ \eqref{e:duality3}\ } \Hom(D^2Z\tens D^2X\tens Y,K)
\]
followed by
\[
(f_Z\tens f_X\tens\id_Y)^* : \Hom(D^2Z\tens D^2X\tens Y,K) \rar{\simeq} \Hom(Z\tens X\tens Y,K).
\]
So property~\eqref{e:ugly} is equivalent to the commutativity of the outer pentagon in the diagram
\[
\xymatrix{
  \Hom(Y\tens Z\tens X,K) \ar[d]_{\eqref{e:duality3}} \ar[rr]^{\eqref{e:duality3}} & & \Hom(D^2X\tens Y\tens Z,K) \ar[d]^{\eqref{e:duality3}} \\
  \Hom(D^2(Z\tens X)\tens Y,K) \ar[dr]_{(f_{Z\tens X}\tens\id_Y)^*} & & \Hom(D^2Z\tens D^2X\tens Y,K) \ar[ll]_{(u_{Z,X}\tens\id_Y)^*} \ar[dl]^{(f_Z\tens f_X\tens\id_Y)^*} \\
   & \Hom(Z\tens X\tens Y,K) &
   }
\]
In this diagram $u_{Z,X}:D^2(Z\tens X)\rar{\simeq} D^2Z\tens D^2 X$ is the isomorphism defining the
monoidal structure on $D^2$. By the definition of $u$ (which was given immediately before
Proposition~\ref{p:DD'}), the top square of the diagram commutes. So the commutativity of the outer
pentagon is equivalent to that of the bottom triangle. The latter is equivalent to $f$ being monoidal
(see Lemma~\ref{l:var_Yoneda}).
\qed

\subsection{Proof of Lemma \ref{l:equivalent-to-pretty1}}\label{ss:proof-l:equivalent-to-pretty1}
Let $\ga$ denote the functorial family of bijections
\[
\ga_{X,Y} : \Hom(X,DY) \rar{\simeq} \Hom(X\tens Y,K) \xrar{\ \ \psi_{X,Y}\ \ } \Hom(Y\tens X,K) \rar{\simeq} \Hom(Y,DX)
\]
where the first and third arrows come from \eqref{e:duality1}. Let $\vp$ be as in
Corollary~\ref{c:abstract}. Comparing Lemma~\ref{l:attempt2}(a) with
Lemma~\ref{l:pivotal2} we see that $\vp_X=f_{D^{-1}X}$.
So by Corollary~\ref{c:abstract},
the condition
\[
\psi_{X,Y}\circ \psi_{Y,X}=\id \quad\quad \forall\, X,Y\in\cM
\]
is equivalent to the condition
\[
f_{D^{-1}X}= D(f_{D^{-2}X})^{-1}  \quad\quad \forall\,X\in\cM \, .
\]
The latter condition holds if and only if $f_{DX}=D(f_X)^{-1}$ for all $X\in\cM$, \; Q.E.D. \qed

\subsection{Proof of Lemma~\ref{l:pivot2}}   \label{ss:proofpivot2}
The isomorphism $\psi_{\e,X}:\Hom (X,K)\rar{\simeq}\Hom (X,K)$ is functorial in $X\in\cM$, so
(b)$\Leftrightarrow$(c) by Yoneda's lemma. To prove that (a)$\Leftrightarrow$(c), it suffices to show that $\psi_{\e,K}=\ga\circ f_K$, where $\ga :D^2K\rar{\simeq}K$ is inverse to \eqref{e:D^2K}.
Diagram \eqref{e:fpsi} for $X=\e$ and $Y=K$ tells us that $\psi_{\e,K}=\ga'\circ f_K$, where
$\ga'\in\Hom (D^2K,K)$ is the image of $\id_K\in\Hom (\e\tens K, K)$ under the isomorphism
\eqref{e:duality3}. But $\ga'=\ga$ by Remark \ref{r:r}(v). \qed

\section{Proof of Proposition \ref{p:monoidal-iso-id-DD}}\label{s:proof-p:monoidal-iso-id-DD}
The idea of the proof is to use the relation between monoidal auto-equivalences and bimodule categories explained in \S\ref{ss:proof_D^2monoidal}.

\mbr

Just as in \S\ref{sss:Application}, we let $\cN$ denote the category of functors
$\cM^\circ\rar{}\sets$,  we write
$\cY:\cM\into\cN$ for the Yoneda embedding and put $n_0:=\cY(K)\in\cN$. We equip $\cN$ with
the $(\cM^\circ,\cM^\circ)$-bimodule structure from \S\ref{sss:Application}. For this structure, $(X\tens n_0)(Z)=\Hom(Z\tens X,K)$ and $(n_0\tens X)(Z)=\Hom(X\tens Z,K)$ for all
$X,Z\in\cM$. In \S\ref{sss:Application} we defined the monoidal structure on $D^2$ using
the isomorphism
\begin{equation}  \label{e:old_friend}
X\tens n_0\rar{\simeq}n_0\tens D^2X ,\quad\quad X\in\cM,
\end{equation}
which comes from \eqref{e:duality3}. On the other hand, the isomorphisms
\[
(\be^\pm_{X,Z})^* : \Hom(Z\tens X,K) \rar{\simeq} \Hom(X\tens Z,K),\quad\quad X,Z\in\cM,
\]
define isomorphisms
\begin{equation}  \label{e:f^pm}
f^\pm_X:X\tens n_0\rar{\simeq}n_0\tens X,\quad\quad X\in\cM \, .
\end{equation}
By \S\ref{sss:Categorification}, each of the isomorphisms \eqref{e:f^pm} defines a monoidal structure
$s^\pm$ on the identity functor $\Id_{\cM}$. By \S\ref{ss:isomorphism_monoidal}, we have a
canonical monoidal isomorphism $\alpha^{\pm}:(\Id_{\cM}\, ,s^\pm)\rar{\simeq}D^2$. Thus to prove
Proposition \ref{p:monoidal-iso-id-DD} it suffices to prove the following lemma.

\begin{lem}
\begin{enumerate}[$($i$)$]
\item The isomorphism $\alpha^{\pm}:\Id_{\cM}\rar{\simeq}D^2$ defined above equals the
isomorphism $\vt^{\pm}:\Id_{\cM}\rar{\simeq}D^2$ from Definition \ref{d:vartheta}.
 \sbr
\item The monoidal structure on $\Id_{\cM}$ induced by the isomorphisms
$f^+_X$ $($resp., $f^-_X${}$)$ is equal to the monoidal structure defined by
\eqref{e:square_of_braiding} $($resp., \eqref{e:inverse_JS}$)$.
\end{enumerate}
\end{lem}

\begin{proof}
To prove (i), we have to show that for each $X\in\cM$, the diagram
\[
\xymatrix{
  X\tens n_0 \ar[rr]^{f^\pm_X} \ar[dr]_{\eqref{e:old_friend}}  & & n_0\tens X \ar[dl]^{\id_{n_0}\tens\vt^\pm_X} \\
   & n_0 \tens D^2(X) &
   }
 \]
commutes; here the bottom left arrow is the isomorphism \eqref{e:old_friend}. This is a diagram
of functors; evaluating them on a test object $Y\in\cM$, we get
the diagram
\[
\xymatrix{
  \Hom(Y\tens X,K) \ar[rr]^{(\be^\pm_{X,Y})^*} \ar[dr]_{\eqref{e:duality3}} & & \Hom(X\tens Y,K) \ar[dl]^{(\vt^\pm_X\tens\id_Y)^*} \\
   & \Hom(D^2X\tens Y,K) &
   }
\]
which commutes by Definition \ref{d:vartheta}.

\mbr

Statement (ii) of the lemma is equivalent to the following easy fact about braided monoidal categories:
if $X,Y,Z\in\cM$, then the square
\begin{equation}\label{e:easy-diagram}
\xymatrix{
 X\tens Y\tens Z \ar[rr]^{\be_{X,Y\tens Z}} \ar[d]_{(\be_{Y,X}\be_{X,Y})\tens\id_Z} & & Y\tens Z\tens X \ar[d]^{\be_{Y,Z\tens X}} \\
 X\tens Y\tens Z \ar[rr]^{\be_{X\tens Y,Z}} & & Z\tens X\tens Y
   }
\end{equation}
commutes. To verify this fact, note that by the hexagon axiom, the diagram
\[
\xymatrix{
 X\tens Y\tens Z \ar[rr]^{\be_{X,Y\tens Z}} \ar[dr]_{\be_{X,Y}\tens\id_Z} & & Y\tens Z\tens X \ar[dr]_{\be_{Y,Z}\tens\id_X} \ar[rr]^{\be_{Y,Z\tens X}} & & Z\tens X\tens Y \\
  & Y\tens X\tens Z \ar[ur]_{\id_Y\tens\be_{X,Z}} \ar[rr]_{\be_{Y\tens X,Z}} & & Z\tens Y\tens X \ar[ur]_{\id_Z\tens\be_{Y,X}} &
   }
\]
commutes. Moreover, the functoriality of $\be$ implies that
\[
(\id_Z\tens\be_{Y,X})\circ\be_{Y\tens X,Z} = \be_{X\tens Y,Z}\circ(\be_{Y,X}\tens\id_Z),
\]
which implies that \eqref{e:easy-diagram} commutes and proves statement (ii).
\end{proof}

\section{From \GV categories to r-categories}   \label{s:GV_to_r}
In  \S\S\ref{ss:M-abstract}--\ref{ss:proof-p:inverse-construction}  we prove Proposition \ref{p:inverse-construction}. In \S\ref{ss:not_automatic} we give an example showing that
in Proposition~\ref{p:inverse-construction}  the condition $Df=f$ does not hold automatically.

\mbr

To prove Proposition~\ref{p:inverse-construction}, we provide a right inverse for the construction from \S\ref{ss:hecke-subcategory-r-category}. Namely, given a \GV category $(\cM',K')$ and
a morphism $f:K'\rar{}\e'$ such that $Df=f$, we construct an r-category $\cM$ and a closed idempotent $e\in\cM$ (this is done in two steps: in \S\ref{ss:M-abstract} we construct $\cM$ as an abstract category, and in \S\ref{ss:M-monoidal}  we define the monoidal structure on $\cM$).
Then we show in Lemmas~\ref{l:hecke-obvious}-\ref{l:rig} that the pair $(\cM , e)$ has the properties required in Proposition~\ref{p:inverse-construction} (in particular, we prove that the monoidal category $\cM$ is an r-category).

\subsection{$\cM$ as an abstract category}\label{ss:M-abstract}
In this subsection we work with abstract categories rather than monoidal ones.
For convenience, we preserve the same notation as above: namely, we consider a category $\cM'$ together with objects $K',\e'\in\cM'$ and a morphism $f:K'\rar{}\e'$. However, we make no assumptions about $K',\e',f$.

\begin{defin}\label{d:abstract-extension}
Let $(\cM,\iota,\e,\de,\pi)$ be the universal\footnote{I.e., initial as an object of the (1-)category of all such data, with the obvious notion of morphism.} datum consisting of a category $\cM$, a functor $\iota:\cM'\rar{}\cM$ and a commutative triangle
\[
\xymatrix{
 & \e \ar[dr]^{\pi} & \\
 \iota(K') \ar[ur]^{\de} \ar[rr]_{\iota(f)} & & \iota(\e')
   }
\]
in $\cM$ (note that $\cM$ is determined uniquely up to isomorphism of categories).
\end{defin}

\begin{lem}\label{l:abstract-extension}
Let $(\cM,\iota,\e,\de,\pi)$ be as in the definition above. Then
\begin{enumerate}[$($i$)$]
\item $\iota$ is fully faithful;
 \sbr
\item for each $X\in\cM'$, the maps
\[
\Hom_{\cM}(\iota(X),\iota(K'))\rar{}\Hom_{\cM}(\iota(X),\e),   \quad g\mapsto \delta\circ g,
\]
and
\[
\Hom_{\cM}(\iota(\e'),\iota(X))\rar{}\Hom_{\cM}(\e,\iota(X)),    \quad g\mapsto g\circ\pi,
\]
 are bijective.
\sbr
\item the map
\[
\Hom (\iota(K'),\iota(\e'))\rar{}\Hom (\e , \e ), \quad g\mapsto \delta\circ g\circ\pi
\]
is injective with image $\Hom (\e ,\e )\setminus \{\id_{\e} \}$.
\end{enumerate}
\end{lem}

\begin{proof}
We first construct a datum  $(\cM,\iota,\e,\de,\pi)$ for which properties (i)-(iii) are obvious
and then check that this datum is universal.

\mbr

The class of objects $Ob(\cM)$ is defined to be
the disjoint union of $Ob(\cM')$ and a one-element set $\{\e\}$.
Define maps $\Phi,\Psi:Ob (\cM)\rar{}Ob(\cM')$ as follows:
\[
\Phi(\e)=\e', \qquad \Psi(\e)=K', \qquad \Phi(X)=\Psi(X)=X\quad\forall\,X\in Ob(\cM').
\]
For all $X\in Ob (\cM)$ let $f_X:\Psi(X)\rar{}\Phi(X)$ be the morphism in $\cM'$ given by
\[
f_{\e}=f, \qquad f_X=\id_X \quad\text{for}\quad X\in Ob(\cM').
\]

\mbr
For $X,Y\in Ob (\cM)$, set $\operatorname{hom}(X,Y)=\Hom_{\cM'}(\Phi(X),\Psi(Y))$. Given
$X,Y,Z\in Ob (\cM)$ and $u\in\operatorname{hom}(X,Y)=\Hom_{\cM'}(\Phi(X),\Psi(Y))$ and $v\in\operatorname{hom}(Y,Z)=\Hom_{\cM'}(\Phi(Y),\Psi(Z))$, set $v\tilde{\circ}u:=v\circ f_Y\circ v\in\operatorname{hom}(X,Z)=\Hom_{\cM'}(\Phi(X),\Psi(Z))$. It is evident that $\tilde{\circ}$ defines an associative composition operation on the sets $\operatorname{hom}(X,Y)$.

\mbr

Now add to $\cM$ one more morphism $\e\rar{}\e$, denoted by $\id_{\e}$, and extend the operation $\tilde{\circ}$ by setting $\id_{\e}\tilde{\circ} u=u$ and $v\tilde{\circ}\id_{\e}=v$ whenever these compositions make sense. Then $\cM$ becomes a category. By construction, $\cM'$ is a full subcategory of $\cM$. Let $\iota:\cM'\into\cM$ be the inclusion functor, let
$\delta\in\operatorname{hom}(K',\e)$ correspond to $\id_{K'}\in\Hom_{\cM'}(K',K')$, and let $\pi\in\operatorname{hom}(\e,\e')$ correspond to $\id_{\e'}\in\Hom_{\cM'}(\e',\e')$.

\mbr

The datum $(\cM,\iota,\e,\de,\pi)$ clearly has properties (i)-(iii). It remains to check that this datum
is universal, i.e., given another datum  $(\overline{\cM},\bar\iota,\bar\e,\bar\de,\bar\pi)$ there is a
unique functor $F:\cM\rar{}\overline{\cM}$ such that
\begin{equation}  \label{e:dano}
F\bigl\lvert_{\cM'}=\bar\iota, \quad F(\e )=\bar\e ,
\end{equation}
\[
F(\de)=\bar\de, \quad F(\pi)=\bar\pi .
\]
If such $F$ exists then one should have
\begin{equation}  \label{e:shouldhave1}
F(\delta \tilde{\circ} g)=\bar\delta\circ\bar\iota (g) \quad\quad\forall g\in\Hom (X,K'), \;X\in\cM' ;
\end{equation}
\begin{equation}  \label{e:shouldhave2}
F( g\tilde{\circ} \pi )=\bar\iota (g)\circ\bar\pi \quad\quad\forall g\in\Hom (\e ,X), \;X\in\cM' ;
\end{equation}
\begin{equation}  \label{e:shouldhave3}
F( \delta \tilde{\circ}g\tilde{\circ} \pi )=\bar\delta\circ\bar\iota (g)\circ\bar\pi
\quad\quad\forall g\in\Hom (K',\e).
\end{equation}

Since $\cM$ has properties (i)-(iii) the action of $F$ on objects and morphisms is uniquely determined by  \eqref{e:dano} and \eqref{e:shouldhave1}-\eqref{e:shouldhave3}.
It is easy to check that the action of $F$ on morphisms
defined by \eqref{e:dano}-\eqref{e:shouldhave3} agrees with composition.
\end{proof}

\subsection{$\cM$ as a monoidal category}\label{ss:M-monoidal}
In this subsection we assume that $\cM'$ is a monoidal category with unit object $\e'$, and that $(K',f)$ is a pair consisting of an object $K'\in\cM'$ and a morphism $f:K'\rar{}\e'$. Let $(\cM,\iota,\e,\de,\pi)$ be as in Definition \ref{d:abstract-extension}.
By Lemma \ref{l:abstract-extension}, $\iota$ is fully faithful, so we will view $\cM'$ as a full subcategory of $\cM$ and
omit the symbol $\iota$ from now on.

\begin{lem}\label{l:monoidal-extension}
Suppose that the following diagram commutes:
\begin{equation}\label{e:diag}
\xymatrix{
 & \e'\tens K' \ar[dr]^{\simeq} & \\
 K'\tens K' \ar[ur]^{f\tens\id_{K'}} \ar[dr]_{\id_{K'}\tens f} & & K' \\
 & K'\tens\e' \ar[ur]_{\simeq} &
   }
\end{equation}
Then there is a unique way to extend the monoidal structure $\tens:\cM'\times\cM'\rar{}\cM'$ to a bifunctor $\tens:\cM\times\cM\rar{}\cM$ so that the following properties are satisfied:
 \sbr
\begin{enumerate}[$($1$)$]
\item the functors $X\longmapsto\e\tens X$ and $X\longmapsto X\tens\e$ are equal to the identity functor on $\cM$;
 \sbr
\item for each $X\in\cM'$, the morphisms
\[
X=\e\tens X \xrar{\ \ \pi\tens\id_X\ \ } \e'\tens X
\quad\text{and}\quad X=X\tens\e\xrar{\ \ \id_X\tens\pi\ \ }X\tens\e'
\]
are equal to the isomorphisms $X\rar{\simeq}\e'\tens X$ and $X\rar{\simeq}X\tens\e'$ that come from the structure of unit object on $\e'\in\cM'$.
\end{enumerate}
 \sbr
Furthermore, there is a unique way to extend the associativity constraint for $\tens$ on $\cM'$ to an associativity constraint for the bifunctor $\tens$ on $\cM$ so that $(\cM,\tens,\e)$ becomes a monoidal category with trivial unit constraints.
\end{lem}

\begin{proof}
The universal property of $\cM$ implies that there is a unique way to define $X\tens Y$ as a functor of $Y\in\cM$ for a fixed $X\in\cM$, and as a functor of $X\in\cM$ for a fixed $Y\in\cM$. It remains to check the commutativity of the diagram
\[
\xymatrix{
  X_1\tens Y_1 \ar[rr]^{\id_{X_1}\tens v} \ar[d]_{u\tens\id_{Y_1}} & & X_1\tens Y_2 \ar[d]^{u\tens\id_{Y_2}} \\
  X_2\tens Y_1 \ar[rr]^{\id_{X_2}\tens v} & & X_2\tens Y_2
   }
\]
for all objects $X_1,X_2,Y_1,Y_2\in\cM$ and morphisms $u:X_1\rar{}X_2$ and $v:Y_1\rar{}Y_2$ in $\cM$. In fact, it suffices to do this when each of $u$ and $v$ is one of the ``standard generators'' of $\cM$ (i.e., is either a morphism in $\cM'$, or $\de:K'\rar{}\e$, or $\pi:\e\rar{}\e'$). The only nontrivial case is where $u=v=\de$ (so that $X_1=Y_1=K'$ and $X_2=Y_2=\e$). Here we are reduced precisely to the commutativity of \eqref{e:diag}.
\end{proof}

The next result is obvious from the construction.

\begin{lem}\label{l:hecke-obvious}
If the assumptions of Lemma \ref{l:monoidal-extension} are satisfied and $\cM$ is equipped with the monoidal category structure described in the lemma, then $\pi$ is an idempotent arrow in $\cM$ and $\cM'$ is identified with the Hecke subcategory $\e'\cM\e'\subset\cM$.  \hfill\qedsymbol
\end{lem}

\begin{rem}     \label{r:univ_property}
The \emph{monoidal} category $\cM$ together with morphisms $K'\rar{\de}\e\rar{\pi}\e'$, constructed above, can also be characterized by a universal property. Namely, suppose $\cN$ is a monoidal category, $\varpi:\e_{\cN}\rar{}e$ is an idempotent arrow, $F:\cM'\rar{}e\cN e$ is a monoidal functor and $\xi:F(K')\rar{}\e_{\cN}$ is a morphism such that the triangle
\[
\xymatrix{
  & \e_{\cN} \ar[drr]^{\varpi} & & \\
  F(K') \ar[ur]^\xi \ar[rr]^{F(f)} & & F(\e') \ar[r]^{\simeq} &  e
   }
\]
commutes. Then $F$ admits a unique extension to a monoidal functor $\cM\rar{}\cN$ such that $\de\mapsto\xi$ and $\pi\mapsto\varpi$.
\end{rem}

\subsection{Proof of Proposition \ref{p:inverse-construction}}\label{ss:proof-p:inverse-construction}
In this subsection we assume that all the hypotheses of Proposition \ref{p:inverse-construction} are satisfied. Let $(\cM,\e,\de,\pi)$ be as in \S\ref{ss:M-monoidal}. The assumption that $Df=f$ is equivalent to the commutativity of the diagram \eqref{e:diag}, and we equip $\cM$ with the monoidal structure described in Lemma \ref{l:monoidal-extension}. In view of Lemma \ref{l:hecke-obvious}, the proof of Proposition \ref{p:inverse-construction} will be complete once we establish

\begin{lem}     \label{l:rig}
Let $D$ denote the unique extension of the duality functor $D:\cM'\rar{\sim}\cM'$ to a contravariant functor $D:\cM\rar{}\cM$ determined by $D(\e)=\e$, $D(\de)=\pi$, $D(\pi)=\de$. Then $D:\cM\rar{}\cM$ is an anti-autoequivalence and there are functorial isomorphisms $\Hom(X\tens Y,\e)\rar{\simeq}\Hom(X,DY)$ for all $X,Y\in\cM$; in particular, $\cM$ is an r-category.
\end{lem}

\begin{proof}
The unique extension of $D^{-1}:\cM'\rar{\sim}\cM'$ to a contravariant functor $\cM\rar{}\cM$ determined by $\e\mapsto\e$, $\de\mapsto\pi$, $\pi\mapsto\de$ is quasi-inverse to $D:\cM\rar{}\cM$; thus $D$ is an anti-autoequivalence of $\cM$.

\mbr

Next, consider the contravariant functors $F_1(X,Y)=\Hom(X\tens Y,\e)$ and $F_2(X,Y)=\Hom(X,DY)$ from $\cM\times\cM$ to the category of sets. By assumption, we have an isomorphism $F_1\bigl\lvert_{\cM'\times\cM'}\rar{\simeq}F_2\bigl\lvert_{\cM'\times\cM'}$, since by construction, $\de:K'\rar{}\e$ identifies $F_1\bigl\lvert_{\cM'\times\cM'}$ with the functor $(X,Y)\mapsto\Hom(X\tens Y,K')$. It is easy to check that this isomorphism extends to a unique isomorphism $F_1\rar{\simeq}F_2$ such that when $Y=\e$, the induced map
\[
\Hom(X,\e)=\Hom(X\tens\e,\e) \rar{\simeq} \Hom(X,D\e) = \Hom(X,\e)
\]
equals the identity, and when $X=\e$, the induced map
\[
\Hom(Y,\e)=\Hom(\e\tens Y,\e) \rar{\simeq} \Hom(\e,DY)
\]
equals $D:\Hom(Y,\e)\rar{}\Hom(D\e,DY)=\Hom(\e,DY)$. This proves the lemma.
\end{proof}

\subsection{An example}    \label{ss:not_automatic}
We will show that in Proposition~\ref{p:inverse-construction}  the condition $Df=f$ does not
hold automatically.

\mbr

Let $k$ be a field, $X=\bigl\{ (x,y)\in\bA^2_k \st xy=0 \bigr\}$, and let $\cM'=\sD(X)=D^b_c(X,\ql)$ be the bounded derived category of constructible $\ell$-adic complexes on $X$ equipped with the usual (derived) tensor product $\tens$. Put $K'=K_X[-2]$, where $K_X$ is the dualizing complex of $X$. Then $(\cM',K')$ is a \GV category, and we claim that there exists a morphism $f:K'\rar{}\e'$ such that $Df\neq f$.

\begin{proof}
Since $\cM'$, equipped with the standard symmetry isomorphisms $M\tens N\rar{\simeq}N\tens M$, is a symmetric monoidal category, it follows that a morphism $f:K'\rar{}\e'$ satisfies $Df=f$ if and only if the corresponding morphism $K'\tens K'\rar{}K'$ is symmetric. Thus we need to check that $\Hom(\bigwedge^2 K',K')\neq 0$. We have $\Hom(\bigwedge^2 K',K')=\Hom\bigl( (\bigwedge^2 K_X)[-2],K_X \bigr)$. Since $\bigwedge^2 K_X$ is concentrated at the singular point $0\in X$, one has
\[
\Hom\bigl( (\bigwedge^2 K_X)[-2],K_X \bigr) = \Hom \bigl( \bigwedge^2(K_X)_0[-2], \ql \bigr),
\]
where $(K_X)_0$ is the stalk of $K_X$ at $0$. But $H^{-1}\bigl((K_X)_0\bigr)=\ql$ and $H^i\bigl((K_X)_0\bigr)=0$ for $i\geq 0$, so
$\Hom \bigl( \bigwedge^2(K_X)_0[-2], \ql \bigr) = \Hom( Sym^2\ql,\ql )=\ql$.
\end{proof}


\end{document}